\documentclass[english,reqno]{article}
\usepackage[T1]{fontenc}
\usepackage[utf8]{inputenc}
\usepackage{lmodern}
\usepackage[left=3.5cm,right=3.5cm,top=2.5cm,bottom=2.5cm]{geometry}

\usepackage[leqno]{amsmath}
\usepackage{amssymb}
\usepackage{amsmath}
\usepackage{amsthm}
\usepackage{pgfplots}
\usepackage{enumitem}

\usepackage{float}
\usepackage{hyperref}

\usepackage{subfig}
\usepackage{graphicx}
\usepackage{todonotes}

\newcounter{lem}
\newtheorem{lemma}[lem]{Lemma}
\newtheorem{remark}[lem]{Remark}
\newtheorem{proposition}[lem]{Proposition}
\newtheorem{theorem}[lem]{Theorem}
\newtheorem{definition}[lem]{Definition}
\newtheorem{assumption}[lem]{Assumptions}
\newtheorem{corollary}[lem]{Corollary}
\newtheorem{notation}[lem]{Notation}

\usepackage{mathtools}
\usepackage{bbm}
\usepackage{bm}
\usepackage[]{algorithm2e}
\usepackage[affil-it]{authblk}

\DeclareMathOperator{\E}{\mathbb{E}}%
\DeclareMathOperator{\Pro}{\mathbb{P}}
\newcommand*\indic[1]{\mathbbm{1}_{\{ #1 \}}}
\newcommand*\indica[1]{\mathbbm{1}_{ #1 }}

\DeclareMathOperator*{\esssup}{ess\,sup}

\newcommand{\mydots}{\cdots}
\def \limhb#1#2#3{\mathrel{\mathop{\kern 0pt#1}\limits_{#2}^{#3}}}

\newcommand{\PM}{\bm{\mathrm{N}}}

\usepackage{xargs}
\newcommandx{\Flow}[2][1=,2=]{  \varphi^{#1}_{#2}}

\newcommandx{\K}[2][1=,2=]{K^{#1}_{#2}}
\newcommandx{\BK}[2][1=,2=]{\bar{K}^{#1}_{#2}}

\renewcommandx{\H}[2][1=,2=]{H^{#1}_{#2}}
\newcommandx{\BH}[2][1=,2=]{\bar{H}^{#1}_{#2}}

\renewcommandx{\r}[2][1=,2=]{r^{#1}_{#2}}

\usepackage{ifmtarg}

\makeatletter

\newcommandx{\Y}[4][1=,2=,3=(a.),4=]{
	Y^{#1\@ifmtarg{#1}{}{\@ifmtarg{#2}{}{,}}#2,#3}_{#4}}
	
\newcommand{\EXP}[1]{\exp{\left( #1 \right) }}
\newcommand*{\RePart}{\mathfrak{R}}
\newcommand*{\ImPart}{\mathfrak{I}}

\begin{document}
	\date{\today}
	\title{Long time behavior of  a mean-field model of interacting neurons}
	\author[1]{Quentin Cormier}
	\author[1]{Etienne Tanr\'e}
	\author[2]{Romain Veltz}
	\affil[1]{Université Côte d’Azur, Inria, France (Tosca team)}
	\affil[2]{Université Côte d’Azur, Inria, CNRS, LJAD, France (MathNeuro team)}

	\maketitle		
\begin{abstract}
We study the long time behavior of the solution to some McKean-Vlasov stochastic differential equation (SDE) driven by a Poisson process. In neuroscience, this SDE models the asymptotic dynamic of the membrane potential of a spiking neuron in a large network.
We prove that for a small enough interaction parameter, 
any solution converges to the unique (in this case) invariant probability measure. To this 
aim, we first obtain global bounds on the jump rate and derive a Volterra type 
integral equation satisfied by this rate. We then replace temporary the 
interaction part of the equation by a deterministic external quantity (we call 
it the external current).
For constant current, we obtain the convergence to the invariant 
probability measure.  Using a 
perturbation method, we extend this result to more general external currents. Finally, we prove the result for the non-linear McKean-Vlasov equation. \\

\end{abstract}

	\section{Introduction}
		We study a model of network of neurons. For each \(N\in\mathbb{N}\),
		we consider a Piecewise-Deterministic Markov Process (PDMP) 
               \(\mathbf{X}^N_t = (X^{1, N}_t, \mydots, X^{N, N}_t) \in\mathbb{R}_+^N\). 
               For  \(i\in\{1,\mydots,N\}\), \(X^{i,N}_t\) models the membrane potential of a neuron (say neuron \(i\)) in the network.
		It emits spikes at random times. The spiking rate of neuron \(i\) at time \(t\) is \(f(X^{i, N}_t)\): it only depends on the potential of neuron \(i\).
		When the neuron \(i\) emits a spike, say at time \(\tau\), its 
		potential is reset (\(X^{i,N}_{\tau+} = 0\))
		and the potential of the other neurons increases by an amount 
		\(\tfrac{J}{N}\), where the connection strength \(J \geq 0\) is fixed:
		\[
		\forall j\neq i, \quad X^{j,N}_{\tau_+} = X^{j,N}_{\tau_-} + \frac{J}{N}.
		\]
		Between two spikes, the potentials evolve according to the one 
		dimensional equation
		\[
			\dfrac{d X^{i, N}_t}{ dt } = b(X^{i, N}_t).
		\]
	The functions \(b\) and \(f\) are assumed to be smooth.
	This process is indeed a PDMP, in particular Markov (see \cite{DavisPDMP}). 
	Equivalently, the model can be described using a system of SDEs driven by Poisson measures.
	Let $(\PM^i(du, dz))_{i=1,\mydots,N}$ be a family of $N$ independent  Poisson measures on $\mathbb{R}_+ \times \mathbb{R}_+$ with intensity measure $du dz$.
	Let $(X_0^{i, N})_{i=1,\mydots,N}$ be a family of $N$ random variables on $\mathbb{R}_+$, \textit{i.i.d.} of law $\nu$ and independent of the Poisson measures.
	Then $(X^{i, N})$ is a \textit{càdlàg} process solution of coupled  SDEs:
	\begin{equation}
	\begin{aligned}
	\forall i,\quad X_t^{i, N} =& X_0^{i, N} + \int_0^t{b(X_u^{i, N}) du} + \frac{J}{N}\sum_{j \neq i}{\int_0^t{\int_{\mathbb{R}_+}{ \indic{z \leq f(X_{u-}^{j, N}) } \PM^j(du, dz)} }}  \\
	               &  - \int_0^t{\int_{\mathbb{R}_+}{X_{u-}^{i, N} \indic{z \leq f(X_{u-}^{i, N})} \PM^i(du, dz)}}.  
	\end{aligned}
	.
	\label{finite_system}
	\end{equation}
	When the number of neurons $N$ goes to infinity, it has been proved 
	(see \cite{de_masi_hydrodynamic_2015,fournier_toy_2016}) for specific linear functions \(b\) and under few assumptions on $f$
	that $X_t^{1, N}$ -  i.e. the first coordinate of the solution to 
	\eqref{finite_system} -  converges  in law to the solution of the  
	McKean-Vlasov SDE:
	\begin{equation}
	X_t = X_0 + \int_0^t{b(X_u) du} + J \int_0^t{\E f(X_u) du}  -  \int_0^t{\int_{\mathbb{R}_+}{X_{u-} \indic{z \leq f(X_{u-})} \PM(du, dz)}},
	\label{NL-equation0}
	\end{equation}
	where, $\mathcal{L}(X_0) := \mathcal{L}(X^{1, N}_0) = \nu$ and $\PM$ is a 
	Poisson measure on $\mathbb{R}_+ \times \mathbb{R}_+$ with intensity 
	measure $du dz$. 
	The measure $\PM$ and $X_0$ are independent.
	
	Equation \eqref{NL-equation0} is a mean-field equation and is the current 
	object of interest. Note that the drift part of \eqref{NL-equation0} 
	involves the law of the solution in the term $\E f(X_u)$: the equation is 
	non-linear in the sense of McKean-Vlasov. Here, we study existence and 
	uniqueness of the solution of \eqref{NL-equation0} and its long time 
	behavior.
	
	Let $\nu(t, \cdot)$ be the law of $X_t$ at time $t \geq 0$. It is a weak solution of the following Fokker-Planck PDE:
	\begin{equation}
	\left\{
	\begin{aligned}
	 \frac{\partial }{\partial t} \nu (t, x) &= -\frac{\partial}{\partial x} [ (b(x) + J r_t) \nu(t, x)] - f(x) \nu(t, x),\quad x>0 \\
\nu(0,\cdot) &= \nu,\quad \nu(t, 0) = \frac{r_t}{b(0) + J r_t}, ~~~r_t = \int_0^\infty{f(x) \nu(t, x) dx}.
	\end{aligned}
	\right.
	\label{eq:fokker_planck_edp}
	\end{equation}
This model with a noisy threshold is known in the physic literature as  the ``Escape Noise'' model (see \cite[Chap. 9]{GerstnetNeuronalDynamic} for references and biological 
considerations). From a mathematical point of view, it has been first studied in
\cite{de_masi_hydrodynamic_2015} and has been the object of further 
developments in \cite{fournier_toy_2016}.  The function \(f:\mathbb{R}_+\to\mathbb{R}_+\) %
can be considered of the type $f(x) = 
(\frac{x}{\vartheta})^\xi$
for large $\xi > 0$ and some soft 
threshold $\vartheta > 0$. In this situation, if the potential of the neuron is 
equal to $x$ then the neuron has a small probability to spike between $t$ and 
$t + dt$ if $x < \vartheta$ and a large probability if $x > \vartheta$. Such a 
choice of $f$ mimics the standard Integrate-And-Fire model with a fixed 
(deterministic) threshold around $\vartheta$.  
	
Results on the existence of solution to~\eqref{NL-equation0}, in a slightly different context (in particular, with \(b(x)\sim_\infty -\kappa x\) for \(\kappa\geq 0\)), have been obtained in~\cite{de_masi_hydrodynamic_2015}: 
the authors explored the case where the initial condition \(\nu\) is 
compactly supported. This property is preserved at any time $t > 0$. So, the 
behavior of the solution with a rate function \(f\) locally Lipschitz 
continuous is similar to the case with a function \(f\) globally Lipschitz continuous.
When the initial condition is not compactly supported, the situation is more 
delicate. In~\cite{fournier_toy_2016}, the authors proved existence and 
path-wise uniqueness of the solution to \eqref{NL-equation0} (in a slightly 
different setting than ours) using an ad-hoc distance.

Note that the global existence results obtained for this model differ from 
those obtained for the ``standard'' Integrate-and-Fire model with a fixed 
deterministic threshold. This situation, studied for instance in 
\cite{PerthameEDP,CarrilloEDP,DIRT1,DIRT2}, corresponds (informally) 
to the choice 
$f(x) = +\infty \indic{x \geq \vartheta}$, $\vartheta > 0$ being the fixed 
threshold. In these papers, a diffusion part is included in the modeling.
In \cite{PerthameEDP}, the 
authors proved that a blow-up phenomenon appears when the 
law of the initial condition is concentrated close to the threshold 
$\vartheta$: the jump rate of the 
solution diverges to infinity in finite time. Here, the situation is completely 
different: the jump rate is uniformly bounded in time (see 
Theorem~\ref{th:existence et unicité de l'équation limite}). In
\cite{PerthameEDP}, the authors have obtained results on the stability of the 
solution for the diffusive model with a deterministic threshold (see also \cite{PierreRoux} for a variant).

Very little is known about the long time behavior of the solutions to \eqref{NL-equation0}. 
One can %
study it by considering the long time behavior of 
the finite particles system \eqref{finite_system} and then apply the 
propagation of chaos to extend the 
results to the McKean-Vlasov equation~\eqref{NL-equation0}. 
This strategy has been developed  
in~\cite{Veretennikov_ergodic_2006,bolley_trend_2010} for diffusive problems. 
The long time behavior of the particles system \eqref{finite_system} has been 
studied in \cite{duarte_model_2014,hodara_non-parametric_2018} (again in a 
slightly different setting but the methods could be adapted to our case): %
the authors proved that the particles system  is 
Harris-ergodic and consequently 
converges weakly to its unique invariant probability measure. 
However, transferring the long time behavior of the particles system to the 
McKean-Vlasov equation is possible if the propagation of chaos holds 
\textit{uniformly in time}. 
In 
\cite{de_masi_hydrodynamic_2015,fournier_toy_2016}, the propagation of chaos is
only proved on compact time interval $[0, T]$ and their 
estimates diverge as $T$ goes to infinity. Because Equation 
\eqref{NL-equation0} may have multiple invariant probability measures, there is no hope in 
general to prove such uniform propagation of chaos.

Coupling methods are also used to study the long time behavior of SDEs. 
In \cite{TCPCoupling}, the authors have studied the TCP (a {linear} PDMP) which is close to \eqref{NL-equation0}.
{The size of the jumps is \(-x / 2\) in the TCP and \(-x\) in our setting, 
	$x$ being the 
	position of the process just before the jump.} The main difference is the non-linearity: we failed to adapt their methods 
when the interactions are non-zero ($J > 0$).

Butkovsky studied in \cite{Butkovsky}  the long time behavior of some McKean-Vlasov diffusion SDE of the form: 
\begin{equation}\label{eq:butkovsky} 
\forall t \geq 0,~X_t = X_0 + \int_0^t{\left[b_1(X_u)  
	+ \epsilon b_2(X_u,  \mu_u)  \right] du} + W_t,  ~\mu_u = \mathcal{L}(X_u), 
\end{equation}
where $(W_t)_{t \geq 0}$ is a Brownian motion. Here the drift terms $b_1$ and $b_2$ are assumed to be globally Lipschitz and $b_2$ is assumed uniformly bounded with respect to its two parameters. The author proved that if the 
parameter $\epsilon$ is small enough, \eqref{eq:butkovsky} has 
a unique invariant probability measure 
which is globally stable. 
The case $\epsilon > 0$ (and small) is treated 
as a perturbation of the case $\epsilon = 0$ using a Girsanov transform. It 
could be interesting to see how this method could be adapted to SDE driven by 
Poisson measures, but we did not pursue this path. 

Another approach consists in studying the non-linear Fokker-Planck equation~\eqref{eq:fokker_planck_edp}. Such non-linear transport equations with boundary 
conditions have been studied in the context of population dynamics (see for 
instance \cite{AgeDepPopDyn,AgeDepPopPruss,AgeDepPopWebb, transportequationbiology}). 
In \cite{AgeDepPopDyn}, the authors 
have characterized the stationary solutions of the PDE and 
found a criterion of local stability for stationary solutions. 
They derived a Volterra integral 
equation and used it to obtain the stability criteria. More recently, 
\cite{Pakdaman2010,Pakdaman2013,MischlerEDP} have re-explored 
these models for neuroscience applications 
(see \cite{Chevallier2015,Chevallier2017} for a rigorous derivation 
of some of these PDEs using Hawkes processes).

PDE \eqref{eq:fokker_planck_edp} differs from theirs in the sense that we 
have a non-linear transport term (theirs is constant and equal to one) and our 
boundary condition is more complex.
The long time behavior of the PDE 
\eqref{eq:fokker_planck_edp} has been successfully studied in 
\cite{fournier_toy_2016} and in \cite{veltz:hal-01290264v2} in the case where 
$b \equiv 0$. In this situation, one can simplify the PDE 
\eqref{eq:fokker_planck_edp} with a simpler boundary condition
\[ 
\nu(t, 0) = \frac{1}{J}.
\]
The authors proved that if the density of the initial condition satisfies this 
boundary condition and regularity assumptions, then $\nu(t, .)$ converges to the 
density of the invariant probability measure as $t$ goes to infinity. The convergence holds 
in $L^1$ or in stronger norms (see \cite{veltz:hal-01290264v2}).
For $b \neq 0$, the boundary condition is more delicate and their methods cannot 
be easily applied.

Actually the long time behavior of the solution to \eqref{NL-equation0} may 
be remarkably intricate. Depending on the choice of $f$, $b$ and $J$, equation \eqref{NL-equation0} may have multiple invariant probability measures. Even if the invariant probability measure is unique, it is not necessarily a stable one and oscillations may appear (see Examples page~\pageref{pageexample}). In 
\cite{drogoul_hopf_2017}, the authors have numerically illustrated this 
phenomenon in a setting close to ours. 

Our main result describes the long time behavior of the solution to
\eqref{NL-equation0} in the weakly connected regime 
(Theorem~\ref{th:Jpetit}). 
If the connection strength $J$ is small enough, we prove that  
\eqref{NL-equation0} has a unique invariant probability measure which is globally stable. 
We give the explicit expression of this non-trivial invariant distribution and 
starting from any initial condition $X_0$, we prove 
the convergence in law of \(X_t\) to it, exponentially fast, as $t$ goes to infinity. 
We argue that this result is very general: it does not depend on the explicit 
shape of the functions $f$ or $b$. 
For stronger connection strengths $J$, such a result cannot hold true in 
general as equation \eqref{NL-equation0} may have multiple invariant probability measures. 

Note that we prove convergence in law, which is weaker than convergence in 
$L^1$. On the other hand, we require very few on the initial condition, in 
particular, we do not assume the existence of a density for the initial 
	condition in Theorem 
	\ref{th:Jpetit}.
We also provide a new proof for the existence and uniqueness of the solution to 
\eqref{NL-equation0}, based on a  Picard iteration scheme (see Theorem 
\ref{th:existence et unicité de l'équation limite}). As in 
\cite{fournier_toy_2016}, we do not require the initial condition to be 
compactly supported.
One of the main difficulty to study \eqref{NL-equation0} (or its PDE version 
\eqref{eq:fokker_planck_edp}) is that there is no simple autonomous equation
for the jump rate $t \mapsto \E f(X_t)$.
To overcome this difficulty, we introduce a ``linearized'' version of 
\eqref{NL-equation0} for which we can derive a closed equation of the jump rate.

Fix a $s \geq 0$ and let $(a_u)_{u \geq s}$ be a continuous deterministic non-negative function, called the external current. It replaces the interaction \(J\mathbb{E}f(X_u)\) in \eqref{NL-equation0}. We consider the  linear non-homogeneous SDE:
	\begin{equation}
	\forall t \geq s,\quad	\Y[s][\nu][(a.)][t] = \Y[s][\nu][(a.)][s] +  \int_s^t{b(\Y[s][\nu][(a.)][u]) du} +  \int_s^t{a_u du} -  \int_s^t{\int_{\mathbb{R}_+}{\Y[s][\nu][(a.)][u-] \indic{z \leq f(\Y[s][\nu][(a.)][u-])} \PM(du, dz)}} ,
	\label{non-homogeneous-SDE}
	\end{equation}
	where $\mathcal{L}(\Y[s][\nu][(a.)][s]) = \nu$. 
	Under quite general assumptions on $b$ and $f$, this SDE has a path-wise unique solution (see Lemma~\ref{lem:existence and uniqueness of the linear non-homogeneous SDE}). We denote the jump rate of this SDE by:
		\begin{equation} \forall t \geq s,\quad \r[\nu][(a.)](t, s) := \E f(\Y[s][\nu][(a.)][t]).
	\label{eq:definition du taux de saut}
	\end{equation}
	Moreover, taking $s = 0$ and $\Y[0][\nu][(a.)][0] = X_0$, it holds that $(\Y[0][\nu][(a.)][t])_{t \geq 0}$  is a solution to \eqref{NL-equation0} if it satisfies the closure condition
	\begin{equation} \forall t \geq 0,\quad a_t = J \r[\nu][(a.)](t, 0).
	\label{eq:closure equation NL-SDE}
	\end{equation}
	Conversely, any solution to \eqref{NL-equation0} is a solution to \eqref{non-homogeneous-SDE} with $a_t = J \E f(X_t)$. 
	We prove that the function $\r[\nu][(a.)]$ satisfies a Volterra integral equation
\begin{equation}
	\forall t \geq s,\quad \r[\nu][(a.)](t, s) = \K[\nu][(a.)](t, s) + \int_s^t{\K[][(a.)](t, u) \r[\nu][(a.)](u, s) du},
	 \label{eq:equation de Volterra ecrite sous forme integrale}
\end{equation}
 where the kernels $\K[\nu][(a.)]$ and $\K[][(a.)]$ are explicit in terms of $\nu$, $a$, $b$ and 
 $f$ (see \eqref{definition de K^x} and \eqref{definition de K^nu et de K}).
	 
Our main tool is this Volterra equation: we use it with a Picard iteration scheme to ``recover'' the non-linear equation \eqref{NL-equation0}.
	The McKean-Vlasov equation \eqref{NL-equation0}, its ``linearized'' 
	non-homogeneous version \eqref{non-homogeneous-SDE}, the Fokker-Planck PDE 
	\eqref{eq:fokker_planck_edp} and the Volterra equation \eqref{eq:equation de Volterra ecrite sous forme integrale} are  different ways to investigate 
	this mean-field problem, each of these interpretations having their own 
	strength and weakness.
Here, we use mainly the Volterra equation \eqref{eq:equation de Volterra ecrite sous forme integrale} and the non-homogeneous SDE \eqref{non-homogeneous-SDE}. To prove that equation \eqref{NL-equation0} has a path-wise unique solution, we rely on the Volterra equation \eqref{eq:equation de Volterra ecrite sous forme integrale} and show that the following mapping:
\begin{equation}
(a_t)_{t \geq 0} \mapsto J \r[\nu][(a.)](., 0) := [t \mapsto J \E f(\Y[0][\nu][(a.)][t]) ], 
\label{eq:fixed point application}
\end{equation}
is contracting on $\mathcal{C}([0, T], \mathbb{R}_+)$ for all $T > 0$.
It then follows that the fixed point of this mapping satisfies the closure 
condition \eqref{eq:closure equation NL-SDE} and can be used to define a 
solution to \eqref{NL-equation0}. Conversely any solution to
\eqref{NL-equation0} defines a fixed point of this mapping and one proves 
strong uniqueness for \eqref{NL-equation0}. 

Finally, we prove our main result concerning the long time behavior of the 
solution to \eqref{NL-equation0}. Let us detail the structure of the proof. 
First, we give in 
Proposition~\ref{prop: convergence du taux de saut dans le cas J=0} 
the long time behavior of the solution to the linear equation 
\eqref{non-homogeneous-SDE} with a constant current (\(a_t \equiv a\)). 
Any solution converges in law to a unique invariant probability measure \(\nu_a^{\infty}\) (Proposition~\ref{prop:mesinvlineaire}). In that case, the Volterra 
equation~\eqref{eq:equation de Volterra ecrite sous forme integrale} is of 
convolution type and it is possible to study finely its solution using Laplace 
transform techniques. 
Second, we prove, for small \(J\), the uniqueness of a constant current 
\(a^*\) such that
\[
\forall t \geq 0,\quad a^* = J \mathbb{E}f( \Y[0][ \nu_{a*}^{\infty} ][a^*][t]).
\]
Third, we extend the previous {convergence} result to non-constant 
currents \((a_t)\) 
satisfying
\begin{equation}  |a_t - a^*| \leq C e^{-\lambda t}.
\label{a_t converge exponentiellement vite vers a}
\end{equation}
Using a perturbation method, we prove that
\[
\Y[0][\nu][(a.)][t]\limhb{\longrightarrow}{t\to \infty}{\mathcal{L}}\nu^\infty_{a^*}.
\]
Fourth, in Theorem~\ref{th:Jpetit}, we give the long time behavior of
the solution to the non-linear equation \eqref{NL-equation0}
for small \(J\). Here, we use a fixed point argument.

The layout of the paper is as follows.
Our main results are given in Section~\ref{sec:notations}. In Section~\ref{sec:technical notations}, we gather technical results. In Section~\ref{sec:existence and uniqueness}, we study the non-homogeneous 
linear equation \eqref{non-homogeneous-SDE} and derive the Volterra equation 
satisfied by the jump rate.
In Section~\ref{sec:the invariant measures}, we characterize the invariant 
probability measures of \eqref{NL-equation0}. 
In Section~\ref{sec: long time behavior for isolated neuron} we study the long time behavior of the solution to~\eqref{non-homogeneous-SDE} with a constant current \(a\).
In Section~\ref{sec: long time behavior for a general drift}, we introduce the perturbation method. 
Finally Section~\ref{sec:proof of the main theorem} is devoted to the proof of our main result (Theorem~\ref{th:Jpetit}).

\section{Notations and results}
\label{sec:notations}
Let us introduce some notations and definitions.  For $s \geq 0$ and a probability measure $\nu$ on $\mathbb{R}_+$, let $\Y[s][\nu][(a.)][s]$ be a $\nu$-distributed random variable, independent of a Poisson measure $\PM(du, dz)$ on $\mathbb{R}_+ \times \mathbb{R}_+$ of intensity measure $du dz$. We consider the canonical filtration $(\mathcal{F}^s_t)_{t \geq s}$ associated to the Poisson measure $\PM$ and to the initial condition $\Y[s][\nu][(a.)][s]$, that is the completion of
\[  \sigma \{\Y[s][\nu][(a.)][s], \PM([s, r] \times A):~ s \leq r \leq t,~ A \in \mathcal{B}(\mathbb{R}_+) \}. 
\]
 
 \begin{definition} Let $s \geq 0$ and consider $(a_t): [s, \infty) \rightarrow \mathbb{R}_+$ a measurable locally integrable function ($\forall t \geq s,~ \int_s^t{a_u du} < \infty$). 
\begin{itemize}
\item A process $(\Y[s][\nu][(a.)][t])_{t \geq s}$ is said to be a solution of the non-homogeneous linear equation \eqref{non-homogeneous-SDE} with a current $(a_t)_{ t \geq s}$ if the law of $\Y[s][\nu][(a.)][s]$ is $\nu$, $(\Y[s][\nu][(a.)][t])_{t \geq s}$ is $(\mathcal{F}^s_t)_{t \geq s}$-adapted,  càdlàg, $a.s. ~\forall t \geq s,~\int_s^t{f(\Y[s][\nu][(a.)][u]) du} < \infty$ and \eqref{non-homogeneous-SDE} holds a.s.
\item An $(\mathcal{F}^0_t)_{t \geq 0}$-adapted càdlàg process $(X_t)_{t \geq 0}$ is said to solve the non-linear SDE \eqref{NL-equation0} if $t \mapsto \E f(X_t)$ is measurable locally integrable and if $(X_t)_{t \geq 0}$ is a solution of \eqref{non-homogeneous-SDE} with $s = 0$, $\Y[0][\nu][(a.)][0] = X_0$ and $\forall t \geq 0,~a_t = J \E f(X_t)$.
\label{def:solution des EDS}
\end{itemize} 
 \end{definition}
 Let $t \geq s \geq 0$. We denote by $\Y[s][\nu][(a.)][t]$ a solution to the linear non-homogeneous SDE \eqref{non-homogeneous-SDE} driven by $(a_t)_{t \geq s} \in \mathcal{C}([s, \infty), \mathbb{R}_+)$ starting with law $\nu$ at time $s$. We denote its associated jump rate by: $\r[\nu][(a.)](t, s) := \E f(\Y[s][\nu][(a.)][t])$. For any measurable function $g$, we write $\nu(g) := \int_0^\infty{g(x) \nu(dx)} = \E g(\Y[s][\nu][(a.)][s])$ whenever this integral makes sense.

Between its random jumps, the SDE \eqref{non-homogeneous-SDE} is reduced to a non-homogeneous ODE. Let us introduce its flow $\Flow[(a.)][t, s](x)$, which by definition is the solution of:
	\begin{align}
		 \label{def:definition of the flow}
	 \forall t \geq s,~\frac{d}{dt} \Flow[(a.)][t, s](x) &= b(\Flow[(a.)][t, s](x)) + a_t \\
	 \Flow[(a.)][s, s](x) &=x. \nonumber
\end{align}	
If $a_t \equiv a$, we denote $\Flow[a][t](x) = \Flow[(a.)][t,0](x)$.
	
\begin{assumption}
\label{assumptions:b0}
We assume that $b: \mathbb{R}_+ \rightarrow \mathbb{R}$ is a  locally Lipschitz function with $b(0) > 0$ and that $b$ is bounded from above:
\begin{equation}
\exists C_b \geq 0:~ \forall x \geq 0, ~b(x) \leq C_b.
\label{assumptions:b0:drift bounded from above}
 \end{equation}
We assume moreover that there is a positive constant $C_\varphi$ such that for all $(a_t)_{t \geq 0}$,  $(d_t)_{t \geq 0} \in \mathcal{C}(\mathbb{R}_+, \mathbb{R}_+)$ 
we have
\begin{equation} \forall x \geq 0,~ \forall s \leq t,~ |\Flow[(a.)][t,s](x) - \Flow[(d.)][t,s](x)| \leq C_\varphi \int_s^t{|a_u - d_u| du}.  
\end{equation}
\label{assumptions:b0:Cvarphi}
\end{assumption}
The assumption \(b(0) > 0\) implies that for all \(x,t,s\in\mathbb{R}_+\), we have \(\Flow[(a.)][t, s](x)\in\mathbb{R}_+\).
 \begin{assumption}
  We assume that $f: \mathbb{R}_+ \rightarrow \mathbb{R}_+$  is a $\mathcal{C}^1$ convex increasing  function with $f(0) = 0$ and satisfies:
 \begin{enumerate}[label=\ref{assumptions:f0}.\arabic*.]
 \item there exists a constant $C_f$ such that
\[ \forall x,y \geq 0,~ f(x+y) \leq C_f (1 + f(x) + f(y)) \text{ and } f'(x+y) \leq C_f (1 + f'(x) + f'(y)).\] \label{assumptions:f0:f(x+y)}
 \item for all $\theta \geq 0,~\sup_{x \geq 0} \{  \theta f'(x)  - f(x) \} < \infty$. 
 
 Define $\psi(\theta) := \sup_{x \geq 0} \{ \theta f'(x) - \frac{1}{2} f^2(x) \} < \infty.$ We also assume that
 \[ \lim_{\theta \rightarrow +\infty}{ \frac{\psi(\theta)}{\theta^2} } = 0.\] \label{assumptions:f0:psi}
 \item Finally we assume that there is a constant $C_{b,f} > 0$ such that
 \[ 
 \forall x \geq 0,~|b(x)| \leq C_{b,f}(1 + f(x)).
 \] 
 \label{assumptions:f0:b}
 \end{enumerate}
 \label{assumptions:f0} 
 \end{assumption}
 Note that these assumptions ensure that $f(x) > 0$ for all $x > 0$.
 \begin{assumption} 
 We assume that the law of the initial condition is a probability measure $\nu$ satisfying  $\nu(f^2) < \infty$.
\label{assumptions:nu0}
\end{assumption}
 Let us give our main results.
  \begin{theorem}
 \label{th:existence et unicité de l'équation limite}
 Under Assumptions~\ref{assumptions:b0},  \ref{assumptions:f0} and \ref{assumptions:nu0}, the non-linear SDE \eqref{NL-equation0} has a  path-wise unique solution $(X_t)_{t \geq 0}$ in the sense of 
 Definition~\ref{def:solution des EDS}.
 Furthermore, there is a finite constant $\bar{r} > 0$ (only depending on $b$, $f$ and $J$)  such that:
   \[ \sup_{t \geq 0} \E f(X_t) \leq \max(\bar{r}, \E f(X_0)) ,~~ \limsup_{t \rightarrow \infty}{\E f(X_t)} \leq \bar{r}. \]
  The upper-bound $\bar{r}$ can be chosen to be an increasing function of $J$.
 \end{theorem}

	\begin{notation}  
		Denote for all $a \geq 0$ the probability measure
		\begin{equation}
		\nu^\infty_{a}(dx) := \frac{\gamma(a)}{b(x) + a} \EXP{-\int_0^x{\frac{f(y)}{b(y) + a} dy}} \indic{x \in [0, \sigma_{a}]} dx,
		\label{eq:invariante measure}
		\end{equation}
		where $\gamma(a)$ is the normalization
		\begin{equation} \gamma(a) := \left[\int_0^{\sigma_a}{  \frac{1}{b(x) + a} \EXP{-\int_0^x{\frac{f(y)}{b(y) + a} dy}} dx}\right]^{-1}. 
		\label{eq:gamma(a) la constante de renormalisation}
		\end{equation}
		The upper bound $\sigma_a$ of the support of \(\nu^\infty_{a}\) is given by
		$\sigma_a := \lim_{t \rightarrow \infty}{\Flow[a][t](0)} \in \mathbb{R}^*_+ \cup \{ +\infty \}$.
		\end{notation}		
		\begin{remark}
			\begin{enumerate}
				\item For all $a \geq 0, ~\gamma(a) = \nu^\infty_a(f)$.
				\item We prove in Proposition~\ref{prop:mesinvlineaire} that for any  $a\geq 0$, $\nu_a^\infty$ is the unique invariant probability measure of \eqref{non-homogeneous-SDE} with $a_t\equiv a$.
			\end{enumerate}
		\end{remark}
		\begin{proposition}\label{prop:mesures invariantes}
		The probability measure \(\nu^\infty_{a}\) is an invariant measure of
		\eqref{NL-equation0} iff
		\begin{equation}
		\frac{a}{\gamma(a)} = J.
		\label{eq:scalar equation invariant measures}
		\end{equation}
		Moreover, define $J_m := \sup\{J_0 \geq 0:~ \forall J \in [0, J_0] \text{ equation \eqref{eq:scalar equation invariant measures} has a unique solution}\}$, then $J_m > 0$. Consequently, for all $ 0 \leq J < J_m$ the non-linear process \eqref{NL-equation0} has a unique invariant probability measure.
\end{proposition}

We now state our main result: the convergence to the unique invariant probability measure for weak enough interactions.
\begin{theorem}
Under Assumptions~\ref{assumptions:b0}, \ref{assumptions:f0}, \ref{assumptions:nu0}, there exists strictly positive constants $J^*$ and $\lambda$ (both only depending on $b$ and $f$) satisfying
\[ 0 < J^* < J_m,~  0 < \lambda < f(\sigma_0),  \]
($J_m$ and $\sigma_0$ are defined in Proposition~\ref{prop:mesures invariantes}) and such that for any $0 \leq J \leq J^*$,  there is a constant $D > 0$:
\[ \forall t \geq 0,~| \E f(X_t) - \gamma(a^*) | \leq D e^{-\lambda t}. \]
Here, $(X_t)_{t \geq 0}$ is the solution of the non-linear SDE \eqref{NL-equation0} starting with law $\nu$ and $a^*$ is the unique solution of \eqref{eq:scalar equation invariant measures}. The constant $D$ only depends on $b$, $f$, $\E f(X_0)$, $J$ and $\lambda$.

Moreover, it holds that $X_t$ converges in law to $\nu^\infty_{a^*}$ at an exponential speed. If $\phi: \mathbb{R}_+ \rightarrow \mathbb{R}$ is a bounded Lipschitz-continuous function, it holds that
\[ 
\exists D' >0, \forall t\geq 0,\quad | \E \phi(X_t) -  \nu^\infty_{a^*}(\phi)| \leq D' e^{-\lambda t} ,
\]
where the constant $D'$ only depends on $b, f, J, \nu, \lambda$ and $\phi$ through its infinite norm and its Lipschitz constant.
\label{th:Jpetit}
\end{theorem}
Note that in Theorem~\ref{th:Jpetit},  the unique invariant probability measure is globally stable: for weak enough interactions, starting from any initial condition, the system converges to its steady state.
\subsection*{Examples}\label{pageexample}
Given the following constants $p \geq 1$, $\mu > 0$ and $\kappa \geq 0$, define, for all $x \geq 0$:
\[ f(x) := x^p, ~b(x) = \mu - \kappa x.\]
Then $(b, f)$ satisfies the Assumptions~\ref{assumptions:b0} and \ref{assumptions:f0}.
In that case, the flow is given by
\[ 
\Flow[(a.)][t,s](x) = xe^{-\kappa(t-s)} + \frac{\mu}{\kappa}[1-e^{-\kappa(t-s)}] + \int_s^t{e^{-\kappa(t-u)}a_u du }. 
\]
We have $\forall x,y \in \mathbb{R}_+,~ f(x+y) \leq 2^{p-1} (f(x) + f(y))$. A similar estimate holds for $f'$. Moreover $\psi(\theta)  = \frac{1}{2} \theta^{\frac{2p}{p+1}}  (p-1)^{\frac{p-1}{p+1}}(1+p)$, so Assumption~\ref{assumptions:f0:psi} holds.
 
 Consequently, Theorem~\ref{th:Jpetit} applies.
When $\kappa > 0$, the invariant probability measures are compactly supported and not necessarily unique. Consider for instance $b(x) = \mu - x$, $f(x) = x^2$. If $\mu$ is small enough, a numerical study shows that there exists $0 < a_1 < a_2 < \infty$ such that the function $a \mapsto \tfrac{a}{\gamma(a)}$ is increasing on $[0, a_1]$, decreasing on $[a_1, a_2]$ and finally increasing on $[a_2, \infty)$.
Thus, if $J \in (a_1, a_2)$, the non-linear equation \eqref{NL-equation0} admits exactly 3 non-trivial invariant probability measures. A numerical study shows that only two of the three are locally stable (bi-stability).

Another interesting example is the following. Assume $b(x) = 2 - 2x$ and $f(x) = x^{10}$.
Then, a numerical study shows that the function $a \mapsto \tfrac{a}{\gamma(a)}$ is increasing on $\mathbb{R}_+$ and consequently for all $J \geq 0$, \eqref{NL-equation0} admits a unique invariant probability measure. But if $J \in [0.7, 1.05]$ a further numerical analysis shows that the law of the solution of \eqref{NL-equation0} asymptotically  oscillates, betraying that the invariant probability measure is not locally stable.
Those examples emphasis on the fact that the condition $J$ small enough is required for Theorem~\ref{th:Jpetit} to hold.
\begin{remark}
Assumption~\ref{assumptions:b0:Cvarphi} is crucial to obtain our result on the long time behavior 
(Theorem~\ref{th:Jpetit}). It restricts us to $\kappa \geq 0$. If $b(x) = \mu - \kappa x$ with 
$\kappa < 0$ then Assumption~\ref{assumptions:b0:Cvarphi} does not hold.
\end{remark}

\section{Technical notations and technical lemmas}
\label{sec:technical notations}
The following standard results on the ODE \eqref{def:definition of the flow} will be useful all along:
\begin{lemma}
\label{lemma:b0}
Assume b satisfies Assumption~\ref{assumptions:b0}. Then:
\begin{enumerate}
\item For all $x \geq 0$ and $s \geq 0$, the ODE \eqref{def:definition of the flow} has a unique solution $t \mapsto \Flow[(a.)][t,s](x)$ defined on $[s, \infty)$. This is the flow associated to the drift $b$ and to the external current $(a_t)_{t \geq 0}$.
\label{lemma:b0:ODE}
\item Given $(a_t)$ and $(d_t)$ in $\mathcal{C}(\mathbb{R}_+, \mathbb{R}_+)$, the flow  satisfies the following comparison principle:
\[  [ \forall t  \geq 0,~ a_t \geq d_t  ]  \implies [\forall x \geq y \geq 0,~\forall t \geq s \geq 0,~   \Flow[(a.)][t,s](x) \geq  \Flow[(d.)][t,s](y)]. \]
\label{lemma:b0:comparison principle}
\item The flow grows at most linearly with respect to the initial condition:
\[ 
\forall a \geq 0,~  ~\forall x \geq 0,~\forall t \geq 0, \quad ~ \Flow[a][t](x) \leq x + C^a_b t,~ \quad \text{ where } \quad C^a_b := C_b + a.  
\] \label{lemma:b0:drift bounded from above}
\item The function $(t, s) \mapsto \Flow[(a.)][t, s](0)$ is continuous on $\{(t,s): 0 \leq s \leq t < \infty\}$.
 \label{lemma:b0:sensibility du flot}
 \item For any constant current $a \geq 0$, the flow converges to a limit as $t$ goes to infinity (possibly equal to $+\infty$):
\begin{equation}
\label{eq:definition de sigma_a} \forall a \geq 0,~ \forall x \geq 0,~ \lim_{t \rightarrow +\infty}{\Flow[a][t](x)} := \sigma^x_a \in \mathbb{R}^*_+ \cup \{ +\infty \}. 
\end{equation}
It holds that $\inf_{a,x \geq 0}{\sigma^x_a} > 0$. Moreover if we define:
\[
\sigma_a := \inf\{x \geq 0:~ b(x) + a = 0 \} \in \mathbb{R}^*_+ \cup \{ +\infty \},  
\]
we have: $\sigma^0_a = \sigma_a$.
\end{enumerate}
\end{lemma}

\begin{remark}
\label{remark:b0f}
\begin{enumerate}
\item  Assumption~\ref{assumptions:f0:psi} ensures that $f$ does not grow too fast in the sense that for all $\epsilon > 0$, there is a constant $C_\epsilon > 0$, such that: $\forall x \geq 0~ f(x) \leq C_\epsilon e^{\epsilon x}$.  \label{remark:b0f:f does not grow too much}
\item Using that $f$ is increasing and continuous, we have, for all $a \geq 0$:
\[ \lim_{t \rightarrow \infty}{ f(\Flow[a][t]) }=f(\sigma_{a}) \geq f(\sigma_{0})>0. \]
\label{remark:b0f:lim f flow a}
\end{enumerate}
\end{remark}

We show that the jump rate $\r[\nu][(a.)]$ of the non-homogeneous SDE \eqref{non-homogeneous-SDE}, satisfies the Volterra equation \eqref{eq:equation de Volterra ecrite sous forme integrale} where the kernels $\K[\nu][(a.)]$ and $\K[][(a.)]$ are defined by
	\begin{align}\label{definition de K^x}
	\forall t \geq s \geq 0,~ \K[\nu][(a.)](t, s) &:= \int_0^\infty {f(\Flow[(a.)][t, s](x)) \EXP{-\int_s^t{f(\Flow[(a.)][u, s](x)) du}} \nu(dx) }, \\
	\label{definition de K^nu et de K}
	 \K[][(a.)](t, s) &:= \K[\delta_0][(a.)](t, s).
	\end{align}
Given two ``kernels'' $\alpha$ and $\beta$, it is convenient to follow the notation of \cite{gripenberg_volterra_1990} and define:
\begin{equation}\label{eq:defconvolution}
\forall t \geq s,~ (\alpha * \beta)(t, s) := \int_s^t{ \alpha(t, u) \beta(u, s) du}. 
\end{equation}
The Volterra equation \eqref{eq:equation de Volterra ecrite sous forme integrale} becomes
\begin{equation}
\r[\nu][(a.)] = \K[\nu][(a.)] + \K[][(a.)] *  \r[\nu][(a.)].
\label{eq:equation de Volterra}
\end{equation}
Similarly to \eqref{definition de K^x} and \eqref{definition de K^nu et de K}, we define the kernels
\begin{equation}
\forall t \geq s,~\H[\nu][(a.)](t, s) := \int_0^\infty{\EXP{-\int_s^t{f(\Flow[(a.)][u, s](x)) du}} \nu (dx) },\quad \H[][(a.)] := \H[\delta_0][(a.)], \quad \forall x \geq 0,~ \H[x][(a.)] := \H[\delta_x][(a.)].
	\label{definition de H^x et de H^nu}
\end{equation}
From the definition, one can check directly the following relation
\begin{equation}
1 * \K[\nu][(a.)] = 1 - \H[\nu][(a.)].
\label{relation entre K et H}
\end{equation}
To shorten notations, we shall also write $\r[][(a.)](t, s) := \r[\delta_0][(a.)](t, s)$.

\noindent When the input current $(a_t)_{t \geq 0}$ is constant and equal to $a$, equation \eqref{non-homogeneous-SDE} is homogeneous and we write
 \[ \forall t \geq 0, ~ \Y[][\nu][a][t] := \Y[0][\nu][a][t],~\r[\nu][a](t) := \r[\nu][(a.)](t, 0), ~\K[\nu][a](t) := \K[\nu][(a.)](t, 0),~\H[\nu][a](t) := \H[\nu][(a.)](t, 0),~ \Flow[a][t](x) := \Flow[(a.)][t, 0](x). \] 
Note that in this homogeneous situation, the operation $*$ corresponds to the classical convolution operation. In particular this operation is commutative  in the homogeneous setting and equation \eqref{eq:equation de Volterra} is a \textit{convolution Volterra equation}.
\begin{remark}
\label{remark:inequality Ha}
For any $(a.) \in \mathcal{C}(\mathbb{R}_+, \mathbb{R}_+)$ and any probability measure $\nu$, we have
\[ \forall t \geq s \geq 0:~\H[\nu][(a.)](t, s) \leq \H[][0](t - s). \]
\end{remark}
\section{Study of the non-linear SDE \eqref{NL-equation0} and of its linearized version \eqref{non-homogeneous-SDE}}
\label{sec:existence and uniqueness}
\subsection{On the non-homogeneous linear SDE \eqref{non-homogeneous-SDE}}

Fix $s \geq 0$ and let $(a_t): [s, \infty) \rightarrow \mathbb{R}_+$ be a continuous function.
We consider the non-homogeneous linear SDE \eqref{non-homogeneous-SDE}.
We always assume that $\nu$, the law of the initial condition $\Y[s][\nu][(a.)][s]$, satisfies Assumptions~\ref{assumptions:nu0}.

\begin{lemma}
Grant Assumptions~\ref{assumptions:b0}, \ref{assumptions:f0} and \ref{assumptions:nu0}.
Then the  SDE \eqref{non-homogeneous-SDE} has a path-wise unique solution on $[s, \infty)$ in the sense of Definition~\ref{def:solution des EDS}.
\label{lem:existence and uniqueness of the linear non-homogeneous SDE}
\end{lemma}
\begin{proof}
We give a direct proof by considering the jumps of $\Y[s][\nu][(a.)][t]$ and by solving the equation between the jumps. 
\begin{itemize}
\item \textbf{Step 1}: we grant Assumptions~\ref{assumptions:b0}, \ref{assumptions:nu0} and assume that $f: \mathbb{R}_+ \rightarrow \mathbb{R}_+$ is measurable and bounded. There exists a constant $0<K < \infty$ such that:
\[ \sup_{x \geq 0} f(x) \leq K. \]
In this case, the solution of \eqref{non-homogeneous-SDE} can be constructed in the following way. Define by induction:
\begin{align*}
\tau_0 &:=  \inf\{t \geq s: ~ \int_s^t{ \int_{\mathbb{R}_+}{ \indic{z \leq f(\Flow[(a.)][u, s](\Y[s][\nu][(a.)][s]))}\PM(du, dz)} } > 0 \}, \\
 \forall n \geq 0,~ \tau_{n+1} &:= \inf\{t \geq \tau_n: ~ \int_{\tau_n}^t{ \int_{\mathbb{R}_+}{ \indic{z \leq f(\Flow[(a.)][u,\tau_n](0))}\PM(du, dz)} } > 0 \}.
\end{align*}
Using that $f \leq K$, it follows that $a.s. \lim_{n \rightarrow \infty}{\tau_n} = +\infty$.
We define:
\[ \Y[s][\nu][(a.)][t] = \Flow[(a.)][t,s](\Y[s][\nu][(a.)][s]) \indica{t \in [s, \tau_0)} + \sum_{n \geq 1}{ \Flow[(a.)][t,\tau_n](0) \indica{t \in [\tau_n, \tau_{n+1})}  } , \]
and we can directly verify that $t \mapsto \Y[s][\nu][(a.)][t]$ is almost surely a solution of \eqref{non-homogeneous-SDE}.

Uniqueness of equation \eqref{non-homogeneous-SDE} follows immediately from Lemma~\ref{lemma:b0} (point  \ref{lemma:b0:ODE}): two solutions have to be equal almost surely before the first jump, from which we deduce that the two solutions have to jump at the same time. By induction on the number of jumps, the two trajectories are almost surely equal.

\item \textbf{Step 2}: 
We now come back to the general case where $f$ is not assumed to be bounded and we adapt the strategy of \cite{fournier_toy_2016}, proof of Proposition 2. We grant Assumptions~\ref{assumptions:b0}, \ref{assumptions:f0} and \ref{assumptions:nu0}.

We use Step~1 with $f^K(x) := f(\min(x, K))$ for some $K > 0$. Let us denote $\Y[s][\nu][(a.),K][t]$ the solution of \eqref{non-homogeneous-SDE} where $f$ has been replaced by $f^K$.
The boundedness of $f^K$ implies the path-wise uniqueness  of $\Y[s][\nu][(a.),K][t]$. We introduce $\zeta_K := \inf\{t \geq 0:~ |\Y[s][\nu][(a.),K][t]| \geq K \}$, it holds that $\Y[s][\nu][(a.),K][t] = \Y[s][\nu][(a.),K+1][t]$ for all $t \in [0, \zeta_K]$ and all $K \in \mathbb{N}$. Moreover, $\zeta_K < \zeta_{K+1}$.
We define $\zeta := \sup_{K}{\zeta_K}$ and deduce the existence and uniqueness of a solution $t \mapsto \Y[s][\nu][(a.)][t]$ of \eqref{non-homogeneous-SDE} on $[0, \zeta[$ such that $\limsup_{t \rightarrow \zeta}{\Y[s][\nu][(a.)][t]} = \infty$ on the event $\{\zeta < \infty \}$.
But any solution of \eqref{non-homogeneous-SDE} satisfies for all $t \geq s,~\Y[s][\nu][(a.)][t] \leq \Flow[(a.)][t,s](\Y[s][\nu][(a.)][s]) < \infty$ a.s. and so it holds that $\zeta = +\infty$ a.s.
\end{itemize}
\end{proof}

\begin{lemma}
Grant Assumptions~\ref{assumptions:b0}, \ref{assumptions:f0} and \ref{assumptions:nu0}. Let $(\Y[s][\nu][(a.)][t])_{t \geq s}$ be the solution of \eqref{non-homogeneous-SDE}. The functions $t \mapsto \E f(\Y[s][\nu][(a.)][t])$, $t \mapsto \E f'(\Y[s][\nu][(a.)][t])$, $t \mapsto \E f'(\Y[s][\nu][(a.)][t])|b(\Y[s][\nu][(a.)][t])|$ and $t \mapsto \E f^2(\Y[s][\nu][(a.)][t])$ are locally bounded on $[s, \infty)$.
Moreover, $t \mapsto \E f(\Y[s][\nu][(a.)][t]) =: \r[\nu][(a.)](t, s)$ is continuous on $[s, \infty)$.
\label{lem: locally bounded and continuous jump rate}
\end{lemma}
\begin{proof}
Consider the interval $[s, T]$ for some $T > 0$. Let $A := \sup_{t \in [s, T]}{a_t}$. It is clear that
\[ \forall t \in [s, T],  a.s. ~\Y[s][\nu][(a.)][t]  \leq \Y[s][\nu][(a.)][s] + \int_s^t{[b(\Y[s][\nu][(a.)][u]) + a_u] du } \leq \Y[s][\nu][(a.)][s] + C_T,  \]
with $C_T := (C_b + A) (T-s)$. We used here that $b$ is bounded from above (Assumption~\ref{assumptions:b0}).
Using that $f^2$ is non-decreasing and Assumption~\ref{assumptions:f0:f(x+y)}, we have:
\[ a.s.\ f^2(\Y[s][\nu][(a.)][t]) \leq f^2(\Y[s][\nu][(a.)][s] + C_T) \leq C_f^2(1 + f(C_T)  + f(\Y[s][\nu][(a.)][s]))^2. \]
Using Assumption~\ref{assumptions:nu0}, we deduce that $t \mapsto \E f^2(\Y[s][\nu][(a.)][t])$ is bounded on $[s, T]$.
By the Cauchy–Schwarz inequality, this implies that  $t \mapsto \E f(\Y[s][\nu][(a.)][t])$ is also bounded on $[s, T]$.
Finally, using the Assumption~\ref{assumptions:f0:psi} (with $\theta$ = 1), there is a constant $C$ such that for all  $x \geq 0~ f'(x) \leq C + f(x)$. Assumption~\ref{assumptions:f0:b} thus yields
\[ \forall x \geq 0 ~f'(x) |b(x)| \leq C_{b,f}(1+f(x))(C+f(x)), \]
and so this proves that $t \mapsto \E f'(\Y[s][\nu][(a.)][t])|b(\Y[s][\nu][(a.)][t])|$ is also bounded on $[s, T]$.
We now apply the Itô formula (see for instance Theorem 32 of \cite[Chap. II]{protter_stochastic_2005}) to $\Y[s][\nu][(a.)][t]$. It gives for any $\epsilon > 0$
\[ f(\Y[s][\nu][(a.)][t+\epsilon]) = f(\Y[s][\nu][(a.)][t]) + \int_t^{t+\epsilon}{f'(\Y[s][\nu][(a.)][u]) [b(\Y[s][\nu][(a.)][u]) + a_u] du } - \int_t^{t+\epsilon}{ \int_0^\infty {f(\Y[s][\nu][(a.)][u-]) \indic{z \leq f(\Y[s][\nu][(a.)][u-])}  \PM(du, dz)}}. \]
Taking the expectation, it follows that
\[ \E f(\Y[s][\nu][(a.)][t+\epsilon]) - \E f(\Y[s][\nu][(a.)][t]) =  \int_t^{t+\epsilon}{\E f'(\Y[s][\nu][(a.)][u])[b(\Y[s][\nu][(a.)][u]) +  a_u] du } - \int_t^{t+\epsilon}{ \E f^2(\Y[s][\nu][(a.)][u]) du }, \]
from which we deduce that $t \mapsto \E f(\Y[s][\nu][(a.)][t])$ is locally Lipschitz and consequently continuous.
\end{proof}
		 \subsection{The Volterra equation}
		 
		 Along this section, we grant Assumptions~\ref{assumptions:b0}, \ref{assumptions:f0} and \ref{assumptions:nu0}. Let $s \geq 0$ and $(a_t)_{t \geq s} \in \mathcal{C}([s, \infty), \mathbb{R}_+)$ be fixed. We consider $(\Y[s][\nu][(a.)][t])_{t \geq s}$ the path-wise unique solution of equation \eqref{non-homogeneous-SDE} driven by the current $(a_t)_{t \geq s}$.
		 Following \cite{fournier_toy_2016}, we define:
		 \[ \tau_{s, t} := \sup\{u \in [s, t]: \Y[s][\nu][(a.)][u]\neq \Y[s][\nu][(a.)][u-]  \}, \]
		 the time of the last jump before $t$, with the convention that $\tau_{s, t} = s$ if there is no jump during $[s, t]$.
		 It follows directly from \eqref{non-homogeneous-SDE} that:
		 \[  \forall t \geq s,~ a.s. ~ \Y[s][\nu][(a.)][t] = \Flow[(a.)][t,s](\Y[s][\nu][(a.)][s]) \indic{\tau_{s, t} = s} + \Flow[(a.)][t,\tau_{s,t}] \indic{\tau_{s, t}  > s}. \]
		 We also define:
		 \[ \forall t \geq s,~ J_t := \int_s^t{\int_0^\infty{ \indic{z \leq f(\Y[s][\nu][(a.)][u-])} \PM(du, dz)} },\]
		 the number of jumps between $s$ and $t$.
		 
		 \begin{lemma} For all $t \geq u \geq s \geq 0$, we have
		 \[  \mathbb{P}(J_t = J_u | \mathcal{F}_u) = \H[ {\Y[s][\nu][(a.)][u] }][(a.)](t, u) ~a.s. \]
		 where $\H[x][(a.)]$ is given by \eqref{definition de H^x et de H^nu}.
		 \label{lem: proba conditionnelle de ne pas avoir de sauts}
		 \end{lemma}
		 \begin{proof}
		 We have $\{J_t = J_u\} = \{ \int_u^t{\int_0^\infty{ \indic{z \leq f(\Y[s][\nu][(a.)][\theta-])} \PM(d\theta, dz)} } = 0 \}$.
		  Moreover, $\mathcal{F}_u$ and  $\sigma\{ \PM([u, \theta] \times A): \theta \in [u, t], A \in \mathcal{B}(\mathbb{R}_+) \}$ are independent.  It follows from the Markov property satisfied by $(\Y[s][\nu][(a.)][t])_{t \geq s}$ that:
		 \[ a.s. ~\mathbb{P}(J_t = J_u | \mathcal{F}_u) = \Phi(\Y[s][\nu][(a.)][u]) \]
		 where: $\Phi(x) := \mathbb{P}( \int_u^t{\int_0^\infty{ \indic{z \leq f(\Flow[(a.)][\theta, u](x))} \PM(d\theta, dz)} } = 0 ) = \H[x][(a.)](t, u)$.		 
		 \end{proof}
		\begin{lemma}[See also \cite{fournier_toy_2016}, Proposition 25]
		For all $t > s$, the law of $\tau_{s, t}$ is given by:
		\[ \mathcal{L}(\tau_{s, t})(du) = \H[\nu][(a.)](t, s) \delta_s(du) + \r[\nu][(a.)](u, s) \H[][(a.)](t, u) \indic{ s < u < t} du. \]
		\label{lem:loi du dernier instant de saut}
		\end{lemma}		 
		\begin{proof}
		First, from Lemma~\ref{lem: proba conditionnelle de ne pas avoir de sauts}, it follows that:
		\[ \mathbb{P}(\tau_{s, t} = s ) = \mathbb{P}(J_t = J_s) = \E (\H[{\Y[s][\nu][(a.)][s]}][(a.)](t, s)) = \H[\nu][(a.)](t, s). \]
		Let now $u \in (s, t]$ and $h>0$ such that: $s < u -h < u \leq t$.
		We have:
		\[ \mathbb{P}(\tau_{s, t} \in ( u -h, u]) = \mathbb{P}( J_u > J_{u - h}, J_t = J_u) =    \mathbb{E}( \indic{J_u > J_{u - h}} \mathbb{P}(J_t = J_u | \mathcal{F}_u)) = \E( \indic{J_u > J_{u-h}} \H[ {\Y[s][\nu][(a.)][u]} ][(a.)](t, u)). \]
		Let $A := \sup_{u \in [s, t]}{a_u}$.
		On the event $\{  J_u > J_{u-h}\}$, the process jumps at least once during $(u-h, u]$ and so, by Lemma~\ref{lemma:b0} (point \ref{lemma:b0:comparison principle}), we have $\Y[s][\nu][(a.)][u] \in [0, \Flow[(a.)][u, u-h] ] \subset [0, \Flow[A][h] ]$. It follows that
		\[ |\mathbb{P}(\tau_{s, t} \in ( u -h, u]) - \E( \indic{J_u > J_{u-h}} \H[][(a.)](t, u)) | \leq \sup_{x \in [0, \Flow[A][h] ]}|\H[x][(a.)](t, u)  - \H[][(a.)](t, u)| \mathbb{P}( J_u > J_{u-h} ). \]
	From the following Lemma~\ref{lem: technical estimate of the jumps}, we have:
	\[ 
	\lim_{h \downarrow 0} \frac{1}{h} \mathbb{P}(J_u > J_{u-h}) = \r[\nu][(a.)](u, s).
	\]		
		Using Lemma~\ref{lemma:b0} (point \ref{lemma:b0:sensibility du flot}),  $x \mapsto \H[x][(a.)](t, u)$ is continuous at $x = 0$. From the continuity of $h \mapsto \Flow[A][h]$ at $h = 0$,  it yields:
		\[ \lim_{h \downarrow 0} \frac{1}{h}|\mathbb{P}(\tau_{s, t} \in ( u -h, u]) - \E( \indic{J_u > J_{u-h}} \H[][(a.)](t, u))| = 0. \]
		Combining the two results, we obtain the stated formula:
		\[ 
		\lim_{h \downarrow 0} \frac{1}{h}\mathbb{P}(\tau_{s, t}  \in ( u -h, u]) = \r[\nu][(a.)](u, s) \H[][(a.)](t, u).
		\]
		This proves the result.
		\end{proof}
		
		\begin{lemma}[See also \cite{fournier_toy_2016}, Lemma 23]
		\label{lem: technical estimate of the jumps}
		For all $u \in (s, t]$ we have:
		\[ \lim_{h \downarrow 0} \frac{1}{h} \mathbb{P}(J_u > J_{u-h}) = \r[\nu][(a.)](u, s). \]
		\end{lemma}
		\begin{proof}
		Again let $A := \sup_{u \in [s, t]}{a_u} < \infty$. We have:
		\begin{align*}
		&\hspace*{-40pt}|  h \r[\nu][(a.)](u, s) - \mathbb{P}(J_u > J_{u-h}) | \\
		 &\leq  \left|h \r[\nu][(a.)](u, s) - h \r[\nu][(a.)](u - h, s) \right|   +  \left| h \r[\nu][(a.)](u - h , s) - \E \int_{u-h}^{u}{f( \Flow[(a.)][\theta, u-h](\Y[s][\nu][(a.)][u-h])) d\theta} \right|  \\
	 &\quad+  \left| \E \int_{u-h}^{u}{f( \Flow[(a.)][\theta, u-h](\Y[s][\nu][(a.)][u-h])) d\theta} - \left[1 - \E \EXP{-\int_{u-h}^{u}{f( \Flow[(a.)][\theta, u-h](\Y[s][\nu][(a.)][u-h])) d\theta}} \right]  \right| \\
		  & =:  \Delta^1_h + \Delta^2_h + \Delta^3_h.
		\end{align*}
		From the continuity of $u \mapsto \r[\nu][(a.)](u , s)$ 
		(Lemma~\ref{lem: locally bounded and continuous jump rate}) it follows that $\lim_{h \downarrow 0}{ \frac{\Delta^1_h } {h}} = 0$. Moreover, 
		\[ 
		\Delta^2_h =  \left|\int_{u-h}^u{ \E f(\Y[s][\nu][(a.)][u-h]) d\theta } - \E \int_{u-h}^{u}{f( \Flow[(a.)][\theta, u-h](\Y[s][\nu][(a.)][u-h])) d\theta} \right|. 
		\]
 	Assumption~\ref{assumptions:b0}  gives
 	\[ \forall y \geq 0, ~\forall \theta \in [u-h, u],~0 \leq \Flow[(a.)][\theta, u-h](y) - y \leq \Flow[A][h](y) - y \leq C^{A}_b h.  \]
 	We deduce that
		  \[  \Delta^2_h  \leq h \int_{u-h}^u{  \E g_h(\Y[s][\nu][(a.)][u-h]) C^{A}_b d\theta },  \]
		with $g_h(x) := \sup_{y \in [0, C^{A}_b h] } f'(x  + y) = f'(x + C^{A}_b h)$. Using Assumption~\ref{assumptions:f0:f(x+y)}, we have $f'(x + C^{A}_b h) \leq C_f (1 + f'(C^{A}_b h) + f'(x))$. \textit{It} follows that $\E g_h(\Y[s][\nu][(a.)][u-h]) \leq  C_f  (1 + f(C^{A}_b h) + \E f'(\Y[s][\nu][(a.)][u-h])) $.
		The function $t \mapsto \E f'(\Y[s][\nu][(a.)][t])$ being locally bounded, we deduce that $\limsup_{h \downarrow 0}{ \frac{\Delta^2_h } {h}} = 0$.
		Finally, using that $\forall x \geq 0,~ |x - (1-e^{-x})| \leq x^2$ we have
		\[
		 \Delta^3_h \leq \E \left( \int_{u-h}^{u}{f( \Flow[(a.)][\theta, u-h](\Y[s][\nu][(a.)][u-h])) d\theta}
		\right) ^2. 
		 \]
		Using the Cauchy-Schwarz inequality, we obtain
		\[ \Delta^3_h  \leq h \E \int_{u-h}^{u}{f^2( \Flow[(a.)][\theta, u-h](\Y[s][\nu][(a.)][u-h])) d\theta} \leq h^2 \E f^2(\Y[s][\nu][(a.)][u-h] + C^{A}_b h).  \]
		Using $\forall x \geq 0, \forall y \in [0, C^A_b t],~ f^2(x + y) \leq C^2_f (1 + f(C^A_b t)+f(x))^2$  (Assumption~\ref{assumptions:f0:f(x+y)}) and the fact that $t \mapsto \E f^2(\Y[s][\nu][(a.)][t])$ and $t \mapsto \E f(\Y[s][\nu][(a.)][t])$ are locally bounded (as seen in the Lemma~\ref{lem: locally bounded and continuous jump rate}), one can find a constant $C_t$ such that
		\[  \Delta^3_h \leq C_t h^2. \]
		 This shows that $\lim_{h \downarrow 0}{ \frac{\Delta^3_h } {h}} = 0$. Combining the three results ends the proof.
		 		\end{proof}

		 \begin{proposition}[See also \cite{fournier_toy_2016}, Theorem 12]
		 \label{prop:formule donnant la loi de X_t}
		 Grant Assumptions~\ref{assumptions:b0}, \ref{assumptions:f0} and \ref{assumptions:nu0} . 
		 Let $s \geq 0$ and $(a_t)_{t \geq s} \in \mathcal{C}([s, \infty), \mathbb{R}_+)$. Let $\Y[s][\nu][(a.)][t]$ be the solution of equation \eqref{non-homogeneous-SDE}, starting from $\mathcal{L}(\Y[s][\nu][(a.)][s]) = \nu$.
		Let $\phi: \mathbb{R}_+ \rightarrow \mathbb{R}_+$ 
		 be a  continuous non-negative function.  It holds that
		 \[
		  \E \phi(\Y[s][\nu][(a.)][t]) =  \int_s^t{\phi(\Flow[(a.)][t,u](0)) \H[][(a.)](t, u)  \r[\nu][(a.)](u, s) du}  + \int_0^{\infty}{ \phi(\Flow[(a.)][t, s](x)) \H[x][(a.)](t, s) \nu(dx)}.
	\]	  
	In particular, $\r[\nu][(a.)](t, s) = \E f(\Y[s][\nu][(a.)][t])$ solves the Volterra equation \eqref{eq:equation de Volterra}
		 	\[ 
		 	\r[\nu][(a.)] = \K[\nu][(a.)] + \K[][(a.)] * \r[\nu][(a.)]. 
		 	\]
		 \end{proposition}
		 \begin{proof}
		 We have, for all $t \geq s$
		 \begin{align*}
		  \E \phi(\Y[s][\nu][(a.)][t]) &=  \E \phi(\Y[s][\nu][(a.)][t]) \indic{\tau_{s, t} = s} +  \E \phi(\Y[s][\nu][(a.)][t] ) \indic{\tau_{s, t} > s} \\
		  & =  \E \phi( \Flow[(a.)][t, s](\Y[s][\nu][(a.)][s])) \indic{\tau_{s, t} = s} +  \E \phi(\Flow[(a.)][t, \tau_{s, t}](0) ) \indic{\tau_{s, t} > s} \\
		  & := \alpha_t + \beta_t. \\
		 \end{align*}
		 Using Lemma~\ref{lem: proba conditionnelle de ne pas avoir de sauts}, it follows that
		 \begin{align*}
		 \alpha_t = \E [ \phi( \Flow[(a.)][t, s](\Y[s][\nu][(a.)][s])) \Pro(J_t = J_s | \mathcal{F}_s) ] = \E [ \phi( \Flow[(a.)][t, s](\Y[s][\nu][(a.)][s])) \H[{\Y[s][\nu][(a.)][s]}][(a.)](t, s) ] = \int_0^\infty{ \phi(\Flow[(a.)][t,s](x)) \H[x][(a.)](t,s) \nu(dx). }
		 \end{align*}
		 Moreover, using Lemma~\ref{lem:loi du dernier instant de saut}, we have $\beta_t = \int_s^t{\phi(\Flow[(a.)][t,u](0)) \r[\nu][(a.)](u, s) \H[][(a.)](t, u) du }.$
		 Taking $\phi = f$ we obtain the Volterra equation \eqref{eq:equation de Volterra}.
		 		 \end{proof}		 		 
Note that using Lemma~\ref{lem:loi du dernier instant de saut}, $\int_s^t{\mathcal{L}(\tau_{s, t})(du)} = 1$ gives:
\[ \H[\nu][(a.)] + \H[][(a.)] * \r[\nu][(a.)] = 1.\]
This last formula is interesting by itself but does not characterize the jump rate $\r[\nu][(a.)]$. We prefer to work with \eqref{eq:equation de Volterra} because, as shown in the next lemma, this Volterra equation admits a unique solution.
\begin{lemma}
\label{lem:existence et unicté de l'équation de Volterra}
Let $s \geq 0$ be fixed, $(a_t)_{t \geq s} \in \mathcal{C}([s, \infty), \mathbb{R}_+)$.
Then equation \eqref{eq:equation de Volterra} has a unique continuous solution $t \mapsto \r[\nu][(a.)](t, s)$ on $[s, \infty)$.
\end{lemma}

\begin{proof}
Fix $T > s$. It is sufficient to prove the existence and uniqueness result on $[s, T]$.
We consider the Banach space $(\mathcal{C}([s, T], \mathbb{R}), ||\cdot||_{\infty, T})$ and define on this space the following operator: $\Gamma: r \mapsto \K[\nu][(a.)] + \K[][(a.)] * r$.
Let $A := \sup_{t \in [s, T]}{a_t}$, we have: $M_s^T= \sup_{s \leq u \leq t \leq T} \K[][(a.)](t, u) < \infty$. This follows from:
\[\forall s \leq u \leq t \leq T,~ \K[][(a.)](t, u) \leq f(\Flow[(a.)][t, u](0)) \leq f(C^{A}_b (T-s)) < \infty. \]
It is clear (using Assumptions~\ref{assumptions:f0:f(x+y)} and \ref{assumptions:nu0}) that the operator $\Gamma: \mathcal{C}([s, T], \mathbb{R}) \rightarrow \mathcal{C}([s, T], \mathbb{R})$ is well defined.
Given $n \in \mathbb{N}$, the iteration $\Gamma^n$ is an affine operator with linear part $\Gamma_0^n: r \mapsto (\K[][(a.)])^{*(n)} * r$. We prove that $\Gamma^n$ is contracting for $n$ large enough, which is equivalent to proving that $\Gamma^n_0$ is contracting for $n$ large enough.
By induction, it is easily shown that
\[ \forall r \in \mathcal{C}([s, T], \mathbb{R}), \forall n \in \mathbb{N}~~ ||\Gamma^n_0(r)||_{\infty,t} := \sup_{u \in [s, t]} |(\Gamma^n_0(r))(u, s)|   \leq \frac{||r||_{\infty, T} (M_s^T (t-s))^n}{n!}.   \]
Consequently $ \forall r \in \mathcal{C}([s, T], \mathbb{R}), \forall n \in \mathbb{N}, || \Gamma^n_0(r)||_{\infty,T} \leq \frac{(M_s^T (T-s))^n}{n!} ||r||_{\infty, T}$ and $\Gamma^n_0$ is contracting for $n$ large enough. We deduce that the operator $\Gamma^n$ is also contracting and has a unique fixed point in $\mathcal{C}([s, T], \mathbb{R})$. It is also a fixed point of $\Gamma$. This proves that \eqref{eq:equation de Volterra} has a unique solution in $\mathcal{C}([s, T], \mathbb{R})$.
\end{proof}
We shall need the following well-known result on Volterra equation:
\begin{lemma}\label{lem:solution a laide de la resolvante.}
Consider $k, w: \mathbb{R}^2_+ \rightarrow \mathbb{R}$ two continuous kernels. The Volterra equation $x = w + k * x$ has a unique solution 
given by $x =w  + r * w$, where $r: \mathbb{R}^2_+ \rightarrow \mathbb{R}$ is the ``resolvent'' of $k$, i.e. the unique solution 
of
\[  
r = k + k * r.
\]
\end{lemma}
\begin{proof}
It is clear from the proof of the preceding lemma that both Volterra equations have a unique solution. Moreover, we have:
$ w+ k * (w + r * w) = w + k * w + (r- k) * w = w + r * w$.
By uniqueness, we deduce that $x = w + r * w$.
\end{proof}
\subsection{The jump rate is uniformly bounded}
\label{subsec:the jump rate is uniformly bounded}
\begin{lemma}
 Grant Assumptions~\ref{assumptions:b0}, \ref{assumptions:f0} and \ref{assumptions:nu0} . 
		 Let $s \geq 0$ and $(a_t)_{t \geq s} \in \mathcal{C}([s, \infty), \mathbb{R}_+)$.
		 Let $\Y[s][\nu][(a.)][t]$ be the solution of equation \eqref{non-homogeneous-SDE}, starting from $\mathcal{L}(\Y[s][\nu][(a.)][s]) = \nu$.
		 Then the functions $t \mapsto \E f'(\Y[s][\nu][(a.)][t])$, $t \mapsto \E f'(\Y[s][\nu][(a.)][t]) b(\Y[s][\nu][(a.)][t])$ and $t \mapsto \E f^2(\Y[s][\nu][(a.)][t])$ are continuous on $[s, \infty)$.
		 \label{lem: continuity of Ef' E f2 etc}
\end{lemma}

\begin{proof}
The proof relies on Proposition~\ref{prop:formule donnant la loi de X_t}. 
Consider the interval $[s, T]$ for some fixed $T > s \geq 0$ and let $A := \sup_{t \in [s, T]}{a_t}$.
Let $\phi \in \{f', f'b, f^2 \}$.
By Lemma~\ref{lemma:b0} (point~\ref{lemma:b0:sensibility du flot}), the function $(t, u) \mapsto \phi(\Flow[(a.)][t,u](0)) \H[][(a.)](t, u) \r[\nu][(a.)](u, s)$ is uniformly continuous on $\{(t, u): s \leq u \leq t \leq T\}$. Consequently
\[ t \mapsto \int_s^t{\phi(\Flow[(a.)][t, u](0)) \H[][(a.)](t, u) \r[\nu][(a.)](u, s)  du} \text{ is continuous on $[s, T]$.} \]
The continuity of $t \mapsto \int_0^\infty{ \phi(\Flow[(a.)][t, s](x)) \H[x][(a.)](t, s) \nu(dx) }$ follows from the Dominated Convergence Theorem. For instance, for $\phi \equiv f'$, one has
\[ 
\forall t \in [s, T], \forall x \geq 0,~ f'(\Flow[(a.)][t,s](x)) \leq f'(\Flow[A][t-s](x)) \leq f'(x + C^{A}_b (t-s)) \leq C_f({f'(x) + 1 + f'(C^{A}_b (T-s))}), 
\] 
from which the result follows easily using Assumption~\ref{assumptions:nu0} and Assumption~\ref{assumptions:f0:psi}.
The same method can be applied for $\phi(x) := f'(x) b(x)$ (using Assumption~\ref{assumptions:f0:b}) and for $\phi(x) := f^2(x)$.
\end{proof}
\begin{proposition}
\label{prop: la constant bar(a)}
Grant Assumptions~\ref{assumptions:b0}, \ref{assumptions:f0} and \ref{assumptions:nu0}. Let $s, J \geq 0$ be fixed.
Given any $\kappa \geq 0$, there is a constant $\bar{a} \geq \kappa$ only depending on $b$, $f$, $J$ and $\kappa$ such that:
\[ \forall (a_t)_{t \geq s} \in \mathcal{C}([s, \infty), \mathbb{R}_+),~ \left\{\sup_{t \geq s}{a_t }\leq \bar{a} \text{ and } J \nu(f) \leq \bar{a} \right\} \implies \sup_{t \geq s}{J \r[\nu][(a.)](t, s)} \leq \bar{a} .\]
Moreover, $\bar{a}$ can be chosen to be an increasing function of $J$ and $\kappa$.
\end{proposition}

\begin{proof}
Assume $\sup_{t \geq s}{a_t} \leq \bar{a}$ for some $\bar{a} > 0$ that we specify later.
Applying the Itô formula and taking expectations yields
\[
\forall t \geq s,~ \E f(\Y[s][\nu][(a.)][t]) = \E f(\Y[s][\nu][(a.)][s]) + \int_s^t{\E f'(\Y[s][\nu][(a.)][u])[b(\Y[s][\nu][(a.)][u]) +  a_u] du } - \int_s^t{ \E f^2(\Y[s][\nu][(a.)][u]) du }.  
\]
Lemma~\ref{lem: continuity of Ef' E f2 etc} implies that  \(t \mapsto \E f(\Y[s][\nu][(a.)][t])\) is \(\mathcal{C}^1\) and
\[ 
\forall t \geq s,~ \frac{d } {dt } \E f(\Y[s][\nu][(a.)][t]) =  \E f'(\Y[s][\nu][(a.)][t])(b(\Y[s][\nu][(a.)][t]) +  a_t) - \E f^2(\Y[s][\nu][(a.)][t]). 
\] 
Using \eqref{assumptions:b0:drift bounded from above}, the Cauchy-Schwarz inequality gives
\begin{align*}
\frac{d } {dt } \E f(\Y[s][\nu][(a.)][t])  & \leq \left\{[ \bar{a} + C_b] \E f'(\Y[s][\nu][(a.)][t]) - \frac{1}{2} \E f^2(\Y[s][\nu][(a.)][t]) \right\} -  \frac{1}{2} \E^2 f(\Y[s][\nu][(a.)][t]) \\
& \leq \frac{1}{2} [2 \psi(\bar{a} +C_b) - \E^2 f(\Y[s][\nu][(a.)][t]) ],
\end{align*} 
where in the last line, we used Assumption~\ref{assumptions:f0:psi}. Setting $M(\bar{a}) := \sqrt{2 \psi(\bar{a} + C_b)}$ and using the sign of the right hand side, we conclude that
\[ \nu(f) \leq M(\bar{a}) \implies [\forall t \geq s ~\E f(\Y[s][\nu][(a.)][t]) \leq  M(\bar{a})].\]
To complete the proof, we need to check that for any $\kappa \geq 0$, any $J \geq 0$, there is a constant $\bar{a} \geq \kappa$ such that $J M(\bar{a}) \leq \bar{a}$. This follows easily from Assumption~\ref{assumptions:f0:psi}, which gives
\[ \lim_{\theta \rightarrow \infty} { \frac{J \sqrt{2 \psi( \theta)}} {\theta} } = 0. \]
It is clear that $\bar{a}(J)$ can be chosen to be a non-decreasing function of $J$ and $\kappa$. We deduce that:
\[
\left[\sup_{t \geq s}{a_t }\leq \bar{a} \text{ and } J \nu(f) \leq \bar{a}\right]  \implies
 \left\{ \begin{aligned} \frac{d } {dt } \E f(\Y[s][\nu][(a.)][t])  & \leq  \frac{1}{2} \left[\frac{\bar{a}^2}{J^2} - \E^2 f(\Y[s][\nu][(a.)][t]) \right] \\
     \E f(\Y[s][\nu][(a.)][s]) & \leq \frac{\bar{a}}{J}.
 \end{aligned} \right\} \implies \sup_{t \geq s}{J \r[\nu][(a.)](t, s)} \leq \bar{a}.
\]
\end{proof}
We have proved that $t \mapsto \E f(\Y[s][\nu][(a.)][t])$ is $\mathcal{C}^1$ and bounded on $\mathbb{R}_+$. The same methods can be applied to the non-linear equation \eqref{NL-equation0}. 
\begin{lemma} Grant Assumptions~\ref{assumptions:b0}, \ref{assumptions:f0} and \ref{assumptions:nu0}. Consider $(X_t)_{t \geq 0}$ a solution of the non-linear equation \eqref{NL-equation0} in the sense of Definition~\ref{def:solution des EDS}. Then $t \mapsto \E f(X_t) \in \mathcal{C}^1(\mathbb{R}_+, \mathbb{R})$ and there is a finite constant $\bar{r} > 0$ (only depending on $b$, $f$ and $J$)  such that:
   \[ \sup_{t \geq 0} \E f(X_t) \leq \max(\bar{r}, \E f(X_0)) ,~~ \limsup_{t \rightarrow \infty}{\E f(X_t)} \leq \bar{r}. \]
   Moreover, $\bar{r}$ can be chosen to be an increasing function of $J$.
\end{lemma}
\begin{proof}
By applying the same argument as in the proof of Lemma~\ref{lem: locally bounded and continuous jump rate} it is clear that the functions
\[ t \mapsto \E f(X_t), t \mapsto \E f'(X_t),  t \mapsto \E f^2(X_t) \text{ and } t \mapsto \E |b(X_t)| f'(X_t) \]
are locally bounded.
Applying the Itô formula and taking expectations yields
\begin{equation}
 \E f(X_t) = \E f(X_0) + \int_0^t{\E f'(X_u) b(X_u) du} + J \int_0^t{\E f'(X_u) \E f(X_u) du} - \int_0^t{\E f^2(X_u) du }. 
 \label{eq:formule d'Ito pour le taux de saut du processus non lineaire}
\end{equation}
We deduce that $t \mapsto \E f(X_t)$ is continuous. Define for all $t \geq 0,~a_t := \E f(X_t)$. From Lemma~\ref{lem:existence and uniqueness of the linear non-homogeneous SDE}, it is clear that:
\[a.s. ~\forall t \geq 0,~  X_t  = \Y[0][\nu][(a.)][t], \]
where $(\Y[0][\nu][(a.)][t])_{t \geq 0}$
 is the solution of \eqref{non-homogeneous-SDE} driven by $(a_t)_{t \geq 0}$.
In particular, Lemma~\ref{lem: continuity of Ef' E f2 etc} applies and the functions $t \mapsto \E f'(X_t)$, $t \mapsto \E f^2(X_t)$ and $t \mapsto \E f'(X_t) b(X_t)$ are continuous.
From equation \eqref{eq:formule d'Ito pour le taux de saut du processus non lineaire}, we deduce that $t \mapsto \E f(X_t) \in \mathcal{C}^1(\mathbb{R}_+, \mathbb{R}_+)$ and
\[
 \frac{d}{dt} \E f(X_t) = \E f'(X_t) b(X_t) + J \E f'(X_t) \E f(X_t) - \E f^2(X_t). 
\]
We have:
\begin{enumerate}
\item \label{ieq:uniform bound NL 1}  $ \!\begin{aligned}[t] \E f'(X_t) b(X_t)  - \frac{1}{4} \E f^2(X_t) \leq \frac{1}{2}[2 C_b \E f'(X_t) - \frac{1}{2} \E f^2(X_t)] \leq \frac{1}{2} \psi(2 C_b)\end{aligned}, $ using Assumptions~\ref{assumptions:b0:drift bounded from above} and \ref{assumptions:f0:psi}.
\item \label{ieq:uniform bound NL 2} $ \!\begin{aligned}[t]
J \E f'(X_t) \E f(X_t) -  \frac{1}{4} \E f^2(X_t) & \leq J \E f'(X_t) \E f(X_t) - \frac{1}{4} \E^2 f(X_t) \\
    & \leq \E f(X_t) [J \E f'(X_t) - \frac{1}{8} \E f(X_t) - \frac{1}{8} \E f(X_t) ]  \\
    & \leq 2 \beta^2, \\
\end{aligned} $ \\
where $\beta := \sup_{x \geq 0}{J f'(x) - \frac{1}{8}f(x)} < \infty$ (by Assumption~\ref{assumptions:f0:psi}).
We used $\sup_{y \geq 0} y(\beta - \tfrac{1}{8} y) \leq 2 \beta^2$ to obtain the last inequality. Note that $\beta$ is a non-decreasing function of $J$.
\end{enumerate}
Combining the points~\ref{ieq:uniform bound NL 1} and \ref{ieq:uniform bound NL 2} gives
\begin{equation}\label{eq:20190424}
\frac{d}{dt} \E f(X_t) \leq \frac{1}{2}[(\psi(2 C_b) + 4 \beta^2) - \E f^2(X_t)].
\end{equation}
We define: $\bar{r} := \sqrt{\psi(2 C_b) + 4 \beta^2}$ and deduce that
\[ \sup_{t \geq 0}{\E f(X_t)} \leq \max(\bar{r}, \E f(X_0)) ,~~ \limsup_{t \rightarrow \infty}{\E f(X_t)} \leq \bar{r}.\]
\end{proof}
\subsection{Existence and uniqueness of the solution of the non-linear SDE: proof of Theorem~\ref{th:existence et unicité de l'équation limite}}
\label{section: existence de l'equation non lineaire}
We now prove that equation \eqref{NL-equation0} has a unique strong solution $(X_t)_{t \geq 0}$.
Let $J > 0$ (the case $J = 0$ has already been treated in 
Lemma~\ref{lem:existence and uniqueness of the linear non-homogeneous SDE} by choosing $(a_t)_{t \geq 0} \equiv 0$).  Let $\nu$, the initial condition, satisfying Assumption~\ref{assumptions:nu0}, be fixed. We grant Assumptions~\ref{assumptions:b0} and \ref{assumptions:f0}.
Let $T > 0$ be a fixed horizon time.
Thanks to Proposition~\ref{prop: la constant bar(a)} with $\kappa := \max(J\E f(X_0), J\bar{r})$, we build the following application:
\begin{equation}
					\begin{array}{ccccc}
			\Phi & : & \mathcal{C}^T_{\bar{a}} & \rightarrow  & \mathcal{C}^T_{\bar{a}}  \\
			& & (a_t)_{t} & \mapsto & J \r[\nu][(a.)](\cdot, 0), \\
			\end{array}
			\label{eq:fixed point application for a fixed horizon time}
\end{equation}			 
where $\mathcal{C}^T_{\bar{a}} := \{ (a_t)_t \in \mathcal{C}([0, T], \mathbb{R}_+):~  \sup_{t \in [0, T]}{a_t} \leq \bar{a} \}$. 
The function  $\r[\nu][(a.)](t, 0) := \E f( \Y[0][\nu][(a.)][t])$ is defined by equation \eqref{non-homogeneous-SDE} (using $s = 0$). 
The constant $\bar{a}$ is given by Proposition~\ref{prop: la constant bar(a)}: in particular $\bar{a}$ does not depend on $T$.
We equip $\mathcal{C}^T_{\bar{a}}$ with the sup norm $||(a_t)_t||_{\infty, T} := \sup_{t \in [0, T]} |a_t|$. The metric space $(\mathcal{C}^T_{\bar{a}}, ||\cdot||_{\infty, T} )$ is complete.
We now prove  that the application $\Phi$ defined by \eqref{eq:fixed point application for a fixed horizon time} is contracting.
Let $(a_t)_t, (d_t)_t \in \mathcal{C}^T_{\bar{a}}$; we denote by $\r[\nu][(a.)](t, s)$ and  $\r[\nu][(d.)](t, s)$ their corresponding jump rate,  where $t$ belongs to $[s, T]$.
Both $\r[\nu][(a.)]$ and $\r[\nu][(d.)]$ satisfy the Volterra equation \eqref{eq:equation de Volterra}.
It follows that the difference $\Delta :=  \r[\nu][(a.)] -\r[\nu][(d.)]$ satisfies:
\begin{align*}
\Delta & = \K[\nu][(a.)] - \K[\nu][(d.)] + \K[][(a.)] * (\r[\nu][(a.)] -\r[\nu][(d.)]) + (\K[][(a.)]-\K[][(d.)])*\r[\nu][(d.)] \\
                              & =  W + \K[][(a.)] * \Delta \text{ with } W := \K[\nu][(a.)] - \K[\nu][(d.)]  + (\K[][(a.)]-\K[][(d.)])*\r[\nu][(d.)] 
\end{align*}
Consequently, $\Delta$ solves the following non-homogeneous Volterra equation with kernel $\K[][(a.)]$
\begin{equation}
\Delta = W + \K[][(a.)] * \Delta.
\label{eq:eq de volterra satisfait par la difference}
\end{equation}
Using Lemma~\ref{lem:solution a laide de la resolvante.}, this equation can be solved explicitly in 
terms of $\r[][(a.)]$, the ``resolvent'' of $\K[][(a.)]$
\begin{equation}
\label{eq:solution eq de volterra Delta}
\Delta = W + \r[][(a.)]  * W.
\end{equation}
\begin{lemma}
There exists a constant $\Theta_T$ only depending on $T$, $f$, $b$ and $\bar{a}$ such that, for all $a, d\in \mathcal{C}^T_{\bar{a}} $:
\[ \forall\ 0 \leq s \leq t \leq T,~ \forall x \in \mathbb{R}_+,~ |\K[\delta_x][(a.)] - \K[\delta_x][(d.)]|(t, s) \leq \Theta_T (1 + f'(x) + f(x) + f'(x) f(x)) \int_s^t{|a_u - d_u| du}. \]
\end{lemma}

\begin{proof}
Fix $(a_t)$ and $(d_t)$ in $\mathcal{C}^T_{\bar{a}} $. We have
\begin{align*}
|\K[\delta_x][(a.)] - \K[\delta_x][(d.)]|(t,s) &= \left|f(\Flow[(a.)][t, s](x))\EXP{-\int_s^t{f(\Flow[(a.)][u, s](x))du}} - f(\Flow[(d.)][t, s](x))\EXP{-\int_s^t{f(\Flow[(d.)][u, s](x))du}} \right| \\
& \leq \left|f(\Flow[(a.)][t, s](x))- f(\Flow[(d.)][t, s](x))\right|\EXP{-\int_s^t{f(\Flow[(a.)][u, s](x))du}}  \\
          & \quad +   f(\Flow[(d.)][t, s](x)) \left|\EXP{-\int_s^t{f(\Flow[(a.)][u, s](x))du}}  -\EXP{-\int_s^t{f(\Flow[(d.)][u, s](x))du}} \right| \\
 &   =: M + N.
\end{align*}
Assumptions~\ref{assumptions:b0} and \ref{assumptions:f0:f(x+y)} together with 
Lemma~\ref{lemma:b0} (\ref{lemma:b0:comparison principle}) give
\begin{align*}
M & \leq |f(\Flow[(a.)][t, s](x))- f(\Flow[(d.)][t, s](x))| \\
    & \leq f'(x + C^{\bar{a}}_b T) |\Flow[(a.)][t, s](x) - \Flow[(d.)][t, s](x)| \\
   & \leq C_f(1+ f'(x) + f'(C^{\bar{a}}_b T)) C_\varphi \int_s^t{|a_u - d_u| du}. 
\end{align*}
Furthermore, using that $\forall A , B \geq 0:~~|e^{-A} - e^{-B}| \leq |A-B|$, we have
\begin{align*}
N & \leq C_f [1 + f(x) + f(C^{\bar{a}}_b T)]  \int_s^t{|f(\Flow[(a.)][u, s](x)) -f(\Flow[(d.)][u, s](x)) |du} \\
   & \leq C_f [1 + f(x) + f(C^{\bar{a}}_b T)] f'(x + C^{\bar{a}}_b T) C_\varphi \int_s^t{ \int_s^r{|a_u - d_u| du} dr} \\
   & \leq T C_\varphi C^2_f  [ 1  + f(x) + f(C^{\bar{a}}_b T)]  [1 + f'(x)  + f'(C^{\bar{a}}_b T)]  \int_s^t{|a_u - d_u| du}. 
\end{align*}
Combining the two estimates, we get the result.
\end{proof}
\begin{proof}[Proof of Theorem~\ref{th:existence et unicité de l'équation limite}]
We now write $\Theta_T$ for any constant that depends only on $T$, on the initial condition $\nu$, on $b$, $f$, $J$ and on $\bar{a}$ and that can change from line to line.
By Assumptions~\ref{assumptions:f0} and \ref{assumptions:nu0}, it follows that:
\[ 
\forall (a_t), (d_t) \in \mathcal{C}^T_{\bar{a}},~\forall t \in [0, T]:~ |\K[\nu][(a.)] -\K[\nu][(d.)]|(t, 0) \leq \Theta_T  \int_0^t{|a_u - d_u| du}. 
\]
Moreover, since   $\sup_{t\in[0,T]}\r[][(d.)](t, 0) \leq \tfrac{\bar{a}}{J}$
by Proposition~\ref{prop: la constant bar(a)}, we have
\[ 
|(\K[][(a.)] - \K[][(d.)])*\r[][(d.)]|(t, 0) = |\int_0^t{(\K[][(a.)] - \K[][(d.)])(t, u) \r[][(d.)](u) du}| \leq \frac{\bar{a}}{J} \Theta_T (1+f'(0)) T \int_0^t{ |a_u - d_u| du}. 
\]
Consequently, there is a constant $\Theta_T$ such that
\[ \forall (a_t),(d_t) \in \mathcal{C}^T_{\bar{a}},~\forall t \in [0, T]: ~|W|(t, 0) \leq \Theta_T \int_0^t{|a_u - d_u| du}.  \]
Using the formula~\eqref{eq:solution eq de volterra Delta}, we deduce that
\begin{align*}
|\Delta(t, 0)| \leq & |W|(t, 0) + \int_0^t{\r[][(a.)](t, u) |W|(u, 0) du} \\
 \leq & \Theta_T (1+T \frac{\bar{a}}{J}) \int_0^t{|a_u - d_u| du}.
\end{align*}
We have proved that there is a constant $\Theta_T$ such that:
\[ 
\forall (a_t), (d_t) \in \mathcal{C}^T_{\bar{a}}, \forall t \in [0, T], \quad || J \r[\nu][(a.)](\cdot, 0) - J\r[\nu][(d.)](\cdot, 0)||_{\infty, t} \leq \Theta_T \int_0^t{||a-d||_{\infty, u} du}.  
\]
This estimate is sufficient to prove Theorem~\ref{th:existence et unicité de l'équation limite} by a classical Picard/Gronwall argument.  We deduce that $\Phi$ has a unique fixed point $(a^*_t)_t$. It is then easy to check that $(\Y[0][\nu][(a^*.)][t])_{t \in [0, T]}$, driven by the current $(a^*)$ and with initial condition $\Y[0][\nu][(a^*.)][0] = X_0$, defines a solution of \eqref{NL-equation0} up to time $T$. This proves existence of a strong solution to \eqref{NL-equation0}.
Now, if $(X_t)_{t \geq 0}$ is a strong solution of \eqref{NL-equation0} in the sense of Definition~\ref{def:solution des EDS}, let $\forall t \geq 0,~a_t := J \E f(X_t)$. We have $\sup_{t \geq 0}{a_t} \leq \max(J \bar{r}, J \E f(X_0)) \leq \bar{a}$ and consequently $(a_t)_{t \in [0, T]} \in \mathcal{C}^T_{\bar{a}}$.
Moreover, it is clear that $(X_t)_{t \geq 0}$ solves \eqref{non-homogeneous-SDE} with $a_t := J \E f(X_t)$ and $\Y[0][\nu][(a.)][0] := X_0$. We deduce that $(a_t)$ is the unique fixed point of $\Phi$: $\forall t \in [0, T]: a_t = a^*_t$. Consequently, by Lemma~\ref{lem:existence and uniqueness of the linear non-homogeneous SDE}, we have: $a.s. ~\forall t \in [0, T] ~X_t = \Y[0][\nu][(a.)][t]$. This proves path-wise uniqueness and ends the proof of Theorem~\ref{th:existence et unicité de l'équation limite}.
\end{proof}
\section{The invariant probability measures: proof of Proposition~\ref{prop:mesures invariantes}}
\label{sec:the invariant measures}
We now study the invariant probability measures of the non-linear process \eqref{NL-equation0}. We follow the strategy of \cite{fournier_toy_2016}: we first study the linear process driven by a constant current $a$ and show that it has a unique invariant  probability measure. We then use this result to study the invariant probability measures of the non-linear equation \eqref{NL-equation0}.
Let $a \geq 0$ and $(\Y[][\nu][a][t])_t$ the solution of the following SDE:
\begin{equation}
\Y[][\nu][a][t] = \Y[][\nu][a][0] + \int_0^t{b(\Y[][\nu][a][u]) du}  + a t -  \int_0^t{\int_{\mathbb{R}_+}{\Y[][\nu][a][u-] \indic{z \leq f(\Y[][\nu][a][u-])} \PM(du, dz)}} 
\label{eq:homogeneous EDS with cst a}
\end{equation}
Equation \eqref{eq:homogeneous EDS with cst a} is equation \eqref{non-homogeneous-SDE} with $\forall t \geq 0,~a_t = a$ and $s = 0$.
\begin{proposition}\label{prop:mesinvlineaire}
Grant Assumptions~\ref{assumptions:b0} and  \ref{assumptions:f0}. Then the SDE \eqref{eq:homogeneous EDS with cst a} has a unique invariant probability measure $\nu^\infty_{a}$ given by equation \eqref{eq:invariante measure}:
\[
\nu^\infty_{a}(dx) := \frac{\gamma(a)}{b(x) + a} \EXP{-\int_0^x{\frac{f(y)}{b(y) + a} dy}} \indic{x \in [0, \sigma_{a}]} dx,
\]
where $\gamma(a)$ is the normalizing factor given by \eqref{eq:gamma(a) la constante de renormalisation}.
Moreover we have $\nu^\infty_a(f) = \gamma(a)$.
\end{proposition}
A proof of this result can be found in \cite[Prop. 21]{fournier_toy_2016} 
with $b(x) := -\kappa x$ and with slightly different assumptions on $f$. We 
give here a proof based on different arguments. Note that the general method 
introduced by  \cite{costa_stationary_1990} to find the stationary measures of 
a PDMP can be applied here; we use a method introduced in this paper to prove 
the uniqueness part.

\begin{proof}
Let us first check that the probability measure $\nu^\infty_a$ is indeed an invariant measure of \eqref{eq:homogeneous EDS with cst a}. \\
\noindent \textbf{Claim 1} The probability measure $\nu^\infty_a$ satisfies Assumption~\ref{assumptions:nu0}.\\
First $b(0) > 0$ yields to $\forall a \geq 0,~ \sigma_a \geq \sigma_0 > 0$.
The function $t \mapsto \Flow[a][t](0)$ is a bijection from $\mathbb{R}_+$ to $[0, \sigma_a)$. Consequently, the changes of variable $x = \Flow[a][t](0)$ and $y = \Flow[a][u](0)$ give
\begin{align*}
\int_0^{\sigma_a}{\frac{f^2(x)}{b(x) + a} \EXP{-\int_0^x{\frac{f(y)}{b(y) + a} dy}} dx} 
 = \int_0^\infty{f^2(\Flow[a][t](0)) \EXP{-\int_0^t{ f(\Flow[a][u](0))du }} dt}.
\end{align*}
This last integral is finite by Remark~\ref{remark:b0f}.

\noindent \textbf{Claim 2} We have: $\K[\nu^\infty_a][a](t) = \gamma(a) \H[][a](t)$. \\
We recall that $\H[][a](t) = \H[ \delta_0 ][a](t, 0)$. We have, for all $t \geq 0$:
\begin{equation}
\K[\nu^\infty_a][a](t) 
= \int_0^{\sigma_a}{ f(\Flow[a][t](x)) \EXP{-\int_0^t{f(\Flow[a][u](x)) du}} \frac{\gamma(a)}{b(x) + a} \EXP{-\int_0^x{\frac{f(y)}{b(y) + a} dy}} dx}.
\end{equation}
The change of variable $y = \Flow[a][u](0)$ yields:
\[ \K[\nu^\infty_a][a](t) = \int_0^{\sigma_a}{ f(\Flow[a][t](x)) \EXP{-\int_0^t{f(\Flow[a][u](x)) du}} \frac{\gamma(a)}{b(x) + a} \EXP{-\int_0^{t(x)}{ f(\Flow[a][u](0))du}} dx}, \]
where $t(x)$ is the unique $t \geq 0$ such that $\Flow[a][t](0) = x$. We now make the change of variable $x = \Flow[a][s](0)$ we obtain (using  the semi-group property satisfied by $\Flow[a][t](0)$):
\begin{align*}
 \K[\nu^\infty_a][a](t) &=  \gamma(a) \int_0^{\infty}{ f(\Flow[a][t](\Flow[a][s](0))) \EXP{-\int_0^t{f(\Flow[a][u](\Flow[a][s](0))) du}} \EXP{-\int_0^{s}{ f(\Flow[a][u](0))du}} ds} \\
 & =  \gamma(a) \int_t^{\infty}{ f(\Flow[a][\theta](0)) \EXP{-\int_0^{\theta}{f(\Flow[a][u](0)) du}} d\theta}  \\
 & = \gamma(a) \left[\H[][a](t) - \lim_{\theta \rightarrow \infty} \EXP{-\int_0^{\theta}{f(\Flow[a][u](0)) du}}\right].
\end{align*}   
Using Remark~\ref{remark:b0f:lim f flow a}, we have: $\lim_{\theta \rightarrow \infty} \EXP{-\int_0^{\theta}{f(\Flow[a][u](0)) du}} = 0$ and the claim is proved.

We now consider $(\Y[][\nu^\infty_a][a][t])_{t \geq 0}$ 
the solution of equation \eqref{eq:homogeneous EDS with cst a} starting from $\mathcal{L}(\Y[][\nu^\infty_a][a][0]) =  \nu^\infty_a$. 
Proposition~\ref{prop:formule donnant la loi de X_t} applies, so $\r[\nu^\infty_a][a](t) = \E f(\Y[][\nu^\infty_a][a][t])$ is the unique solution of the Volterra equation
\[ \r[\nu^\infty_a][a] = \K[\nu^\infty_a][a] + \K[][a] * \r[\nu^\infty_a][a].  \]
Using Claim~2 and the relation \eqref{relation entre K et H}, we verify that the constant function $\gamma(a)$ is a solution of
\[ 
\K[\nu^\infty_a][a] + \K[][a] * \gamma(a)= \gamma(a)\H[][a] + \gamma(a)(1 - \H[][a])  = \gamma(a).
\]
 By uniqueness (Lemma~\ref{lem:existence et unicté de l'équation de Volterra}), 
 we deduce that $\forall t \geq 0,~\r[\nu^\infty_a][a](t) = \gamma(a)$.

Finally, let $\phi: \mathbb{R}_+ \rightarrow \mathbb{R}_+$ be a  measurable function. Using  Proposition~\ref{prop:formule donnant la loi de X_t}, we have:
\begin{align*}
\E \phi(\Y[][\nu^\infty_a][a][t]) & = \gamma(a) \int_0^t{ \phi(\Flow[a][t-u](0)) \H[][a](t-u) du } + \int_0^\infty{\phi(\Flow[a][t](x)) \H[x][a](t) \nu^\infty_a(dx) } \\
& =  \gamma(a) \int_0^t{ \phi(\Flow[a][u](0)) \H[][a](u) du }  \\&\quad + \int_0^{\sigma_a}{ \phi(\Flow[a][t](x)) \EXP{-\int_0^t{f(\Flow[a][u](x)) du}} \frac{\gamma(a)}{b(x) + a} \EXP{ -\int_0^x{\frac{f(y)}{b(y) + a} dy}} dx }.
\end{align*}
The change of variables $y = \Flow[a][u](0)$ and $x = \Flow[a][\theta](0)$ yields
\begin{align*}
\E \phi(\Y[][\nu^\infty_a][a][t]) & = \gamma(a) \int_0^t{ \phi(\Flow[a][u](0)) \H[][a](u) du } \\&\quad +  \gamma(a) \int_0^{\infty}{ \phi(\Flow[a][t](\Flow[a][\theta](0))) \EXP{-\int_0^t{f(\Flow[a][u](\Flow[a][\theta](0))) du}}  \EXP{ -\int_0^\theta{f(\Flow[a][u](0)) du}} d\theta } \\
& =  \gamma(a) \int_0^t{ \phi(\Flow[a][u](0)) \H[][a](u) du } +  \gamma(a) \int_t^{\infty}{ \phi(\Flow[a][u](0)) \EXP{-\int_0^u{f(\Flow[a][\theta](0)) d\theta}}  du } \\
& =  \gamma(a) \int_0^\infty{ \phi(\Flow[a][u](0)) \H[][a](u) du } \\
& = \nu^\infty_{a} (\phi).
\end{align*}
This proves that $\forall t \geq 0,~ \mathcal{L}(\Y[][\nu^\infty_a][a][t]) = \nu^\infty_a$ and consequently $\nu^\infty_a$ is an invariant probability measure of \eqref{eq:homogeneous EDS with cst a}.  Moreover, we have
\[ 
\nu^\infty_a(f) = \gamma(a) \int_0^{\sigma_a}{\frac{f(x)}{b(x) + a} \EXP{-\int_0^x{\frac{f(y)}{b(y) + a} dy}}  dx } = \gamma(a). 
\]
It remains to prove that the  invariant probability measure is unique. 
Following \cite{DavisPDMP} and \cite{costa_stationary_1990}, we define 
$\mathcal{B}^{ac}(\mathbb{R}_+)$ the set of bounded function 
$h: \mathbb{R}_+ \rightarrow \mathbb{R}$ such that for all $x \in \mathbb{R}_+$, 
the function $t \mapsto h(\Flow[a][t](x))$ is absolutely continuous on $\mathbb{R}_+$. 
For $h \in \mathcal{B}^{ac}(\mathbb{R}_+)$, we define $\mathcal{H}h (x) := \left. \tfrac{d}{dt} h(\Flow[a][t](x)) \right|_{t = 0}.$

\noindent \textbf{Claim 3} Let $h \in \mathcal{B}^{ac}(\mathbb{R}_+)$, then for all $x \geq 0$ we have
\[ 
\left. \tfrac{d}{dt} \E h(\Y[\delta_x][][a][t] ) \right|_{t = 0} = \mathcal{L} h (x) \quad \text{ with } \quad \mathcal{L} h (x) := \mathcal{H} h(x) + (h(0) - h(x)) f(x) . 
\]
Let $\tau_1^x = \inf\{t \geq 0:~ \Y[\delta_x][][a][t] \neq \Y[\delta_x][][a][t-] \}$ and $\tau_2^x = \inf\{t \geq \tau_1^x:~ \Y[\delta_x][][a][t] \neq \Y[\delta_x][][a][t-] \}$ be  the times of the first and second jumps of $(\Y[\delta_x][][a][t])$.
We have
\[ \E h( \Y[\delta_x][][a][t] )  = \E h(\Y[\delta_x][][a][t] ) \indic{t < \tau_1^x}  + \E h(\Y[\delta_x][][a][t] ) \indic{\tau_1^x \leq t < \tau_2^x} + \E h(\Y[\delta_x][][a][t] ) \indic{t \geq \tau_2^x} =: \alpha_t + \beta_t  + \theta_t.  \]
By Lemma~\ref{lem: proba conditionnelle de ne pas avoir de sauts}, we have $\alpha_t = h(\Flow[a][t](x)) \mathbb{P}(t < \tau_1^x) = h(\Flow[a][t](x)) \H[x][a](t)$. It follows that $\left. \tfrac{d}{dt} \alpha_t \right|_{t = 0} = \mathcal{H} h(x) - h(x) f(x)$.
Moreover using that the density of $\tau_1^x$ is $s \mapsto \K[x][a](s)$ it holds that $\beta_t = \int_0^t{ h(\Flow[a][t-s](0)) \K[x][a](s) \H[0][a](t-s) ds}$. We deduce that $\left. \tfrac{d}{dt} \beta_t \right|_{t = 0} = h(0) f(x)$. 
Then, using that $h$ is bounded, we have $\theta_t \leq ||h||_\infty \int_0^t{ \int_0^t{ \K[x][a](u) \K[0][a](s-u) du ds} } \in \mathcal{O}(t^2)$. This proves Claim 3.

Let $g$
be a bounded measurable  function. We follow the method of \cite{costa_stationary_1990} (proof of Theorem~3(a)) and define
\[  
\forall x \geq 0, \quad \lambda_g(x) := \int_0^\infty{ g(\Flow[a][t](x)) \EXP{-\int_0^t{f(\Flow[a][r](x)) dr} } dt}. 
\]
\noindent  \textbf{Claim 4}  The function \(\lambda_g\) belongs to $\mathcal{B}^{ac}(\mathbb{R}_+)$ and satisfies $\mathcal{H} \lambda_g (x) = f(x) \lambda_g(x) - g(x)$. \\
Using the semi-group property of $\Flow[a][t](x)$ we have
\[ 
\lambda_g( \Flow[a][t](x)) = \EXP{ \int_0^t{f( \Flow[a][u](x)) du}} \left[ \lambda_g(x) - \int_0^t{g(\Flow[a][u](x)) \EXP{-\int_0^u{f(\Flow[a][\theta](x)) d\theta}} du } \right] . 
\]
This proves that \(\lambda_g\) is in $\mathcal{B}^{ac}(\mathbb{R}_+)$ with $\tfrac{d}{dt} \lambda_g(\Flow[a][t](x)) = f(\Flow[a][t](x)) \lambda_g(\Flow[a][t](x)) - g(\Flow[a][t](x))$ and gives the stated formula.

Consider now $\nu$ an invariant  probability measure with $\nu(f) < \infty$. 
The Markov property at time $t = 0$ together with Claim~3 shows that $\left. \tfrac{d}{dt} \E \lambda_g(\Y[][\nu][a][t]) \right|_{t = 0} = \left. \tfrac{d}{dt} \int_0^\infty{  \E \lambda_g(\Y[][\delta_x][a][t])  \nu(dx) }  \right|_{t = 0} = \nu(\mathcal{L} \lambda_g).$ The exchange of the derivative at time $t = 0$ and the integral on $\mathbb{R}_+$ is legitimate thanks to the Dominated Convergence Theorem. 
Claim~4 and the fact that $\nu$ is an invariant measure then show that
\[ 
0 = \left. \frac{d}{dt} \E \lambda_g(\Y[][\nu][a][t]) \right|_{t = 0}=  \lambda_g(0) \nu(f) - \nu(g). 
\]
The same computations can be done with $g \equiv 1$, giving $\lambda_1(0) \nu(f) = 1$. It follows that
\[ \nu(g) = \frac{\lambda_g(0)}{\lambda_1(0)} = \int_0^\infty{g(x) \nu^\infty_a(dx) }.  \]
We deduce that necessarily  $\nu = \nu^\infty_a$.
\end{proof}
The next lemma characterizes the invariant probability measures of \eqref{NL-equation0}.

\begin{lemma}
The invariant probability measures of the non-linear equation \eqref{NL-equation0} are 
\(\{ \nu^{\infty}_{a} ~|~ a = J \gamma(a), a \in \mathbb{R}_+\}\).
\end{lemma}
\begin{proof}
Let $\nu$ be an invariant probability measure of \eqref{NL-equation0}
and \(\mathcal{L}(X_0) = \nu\). We have
\[ \forall t \geq 0,~\E f(X_t) = \nu(f) =: p. \]
Let $a :=J p$. 
The process $(X_t)_{t \geq 0}$ solves \eqref{eq:homogeneous EDS with cst a} and $\nu$ is an invariant probability measure of equation \eqref{eq:homogeneous EDS with cst a}. It implies that $\nu = \nu^\infty_a$. Moreover $p = \gamma(a)$ and so necessarily $\frac{a}{\gamma(a)} = J$.

Conversely, let $a \geq 0$ such that $\frac{a}{\gamma(a)} = J$.
Let $(\Y[][\nu^\infty_a][a][t])$ be the solution of \eqref{eq:homogeneous EDS with cst a} with \(\mathcal{L}(\Y[][\nu^\infty_a][a][0]) = \nu^\infty_a\).
We have seen that $\E f(\Y[][\nu^\infty_a][a][t]) = \gamma(a)$, it follows that $a =  J \E f(\Y[][\nu^\infty_a][a][t])$. Consequently $(\Y[][\nu^\infty_a][a][t])_{t \geq 0}$ solves \eqref{NL-equation0} and $\nu^\infty_a$ is one of its invariant probability measure.
\end{proof}
The problem of finding the invariant probability measures of the mean-field equation \eqref{NL-equation0} has been reduced to finding the solutions of the scalar equation \eqref{eq:scalar equation invariant measures}. When $J$ is small enough, we can prove that it has a unique solution, which concludes the proof of Proposition~\ref{prop:mesures invariantes}.
\begin{lemma}
Equation \eqref{eq:scalar equation invariant measures} has at least one solution $a^* > 0$. Moreover, there is a constant $J_0 > 0$ such that for all $J \in [0,J_0]$ \eqref{eq:scalar equation invariant measures} has a unique solution.
\end{lemma}
\begin{proof}
Recall \eqref{eq:gamma(a) la constante de renormalisation}.
By the changes of variable $y  = \Flow[a][u](0)$ and $x = \Flow[a][t](0)$, it holds that
\begin{equation}
  \gamma(a)^{-1} = \int_0^\infty{ \EXP{-\int_0^t{f(\Flow[a][u](0)) du}} dt}. 
  \label{eq:equation gamma a avec le flot}
  \end{equation}
In particular, the function $a \mapsto \gamma(a)$ is non-decreasing. 
Furthermore, using that \(b(x) \leq C_b\), we have
\begin{align*}
\frac{a}{\gamma(a)} &\geq a \int_0^\infty{ 
	\EXP{-\int_0^t{f( (a+C_b)u) du}} dt}\\
& \geq  \frac{a}{a+C_b}\int_0^\infty\EXP{-\frac{1}{a+C_b}\int_0^\theta f(z)dz}d\theta.
\end{align*}
We deduce that $\lim_{a \rightarrow +\infty}{a \gamma(a)^{-1}} = +\infty$.
Let $U(a) := a \gamma(a)^{-1}$. One has $U(0) = 0$, $\lim_{a \rightarrow +\infty}{U(a)} = +\infty$ and $U$ is continuous  on $\mathbb{R}_+$. It follows that the equation $U(a) = J$ has at least one solution $a^*$. Moreover, one can show that the function $U$ has a derivative at $a = 0$ and $U'(0) = 1/\gamma(0) > 0$. 
Consequently, there is $a_0 > 0$ such that $U$ is strictly increasing on $[0, a_0]$.
Using $\lim_{a \rightarrow +\infty}{U(a)} = +\infty$, we can find $a_1$ such that: $\forall a \geq a_1, U(a) \geq 1$.
Finally let $J_0 := \min_{a \in [a_0, a_1]} U(a) > 0$. Let $J < J_0$, it is clear that the equation $U(a) = J$ has exactly one solution $a^* \in [0, a_0]$.
\end{proof}

\section{The convergence of the jump rate implies the convergence in law of the time marginals}
\label{sec: convergence of the jumoing rate implies convergence of the law}

The goal of this section is to prove that controlling the behavior of the jump rate $t \mapsto \E f(X_t)$ can be sufficient to deduce the asymptotic law of $(X_t)$, solution of \eqref{NL-equation0}.
  \begin{proposition} \label{prop:convergence en loi vers la mesure invariante}
  Grant Assumptions~\ref{assumptions:b0}, \ref{assumptions:f0}, \ref{assumptions:nu0}.
Let $(X_t)_{t \geq 0}$ be the solution of the non-linear equation \eqref{NL-equation0}. Assume that there exist constants $\lambda, C > 0$ and $a^* \geq 0$ (that may depend on $b$, $f$, $\nu$, and $J$) such that:
\[ \forall t \geq 0,~ |\E f(X_t) - \gamma(a^*)| \leq C e^{-\lambda t}, \]
and that $a^*$ satisfies equation \eqref{eq:scalar equation invariant measures}: $\frac{a^*}{\gamma(a^*)} = J$. Then
\[ X_t \limhb{\longrightarrow}{t\to \infty}{\mathcal{L}}\nu^\infty_{a^*}. \]
Moreover, if $\phi: \mathbb{R}_+ \rightarrow \mathbb{R}$ is any bounded Lipschitz-continuous function, it holds that
\[ \forall 0 < \lambda' < \min(\lambda, f(\sigma_0)), ~\exists D >0 
\quad\text{s.t.}\quad \forall t\geq 0,\quad  | \E \phi(X_t) -  \nu^\infty_{a^*}(\phi)| \leq D e^{-\lambda' t} ,\]
where the constant $D$ only depends on $b, f, J, C, \nu, \lambda'$ and $\phi$ through its infinite norm and its Lipschitz constant.
\end{proposition}

\begin{proof}
Let $(X_t)_{t \geq 0}$ be the solution of \eqref{NL-equation0} and 
$\phi: \mathbb{R}_+ \rightarrow \mathbb{R}$  a bounded Lipschitz-continuous function, with Lipschitz constant $l_\phi$. 
Consider $\lambda' \in (0, \min(\lambda, f(\sigma_0)))$. 
We denote by $D$ any constant only depending on  $b, f, J, C, \nu, \lambda'$, $||\phi||_\infty$ and $l_\phi$ which shall change from line to line. 
Define for all $t \geq 0,~a_t := J \E f(X_t)$. 
It holds that $(X_t)_{t \geq 0}$ is a solution of \eqref{non-homogeneous-SDE}
with driving current \((a_.)\).
Denote $\r[\nu][(a.)](t, 0) = \E f(X_t)$. By Proposition~\ref{prop:formule donnant la loi de X_t}, we have
\[
\E \phi(X_t) = 
	\int_0^t{\phi(\Flow[(a.)][t,u](0)) \H[][(a.)](t, u)  \r[\nu][(a.)](u, 0) du}
+ 
	\int_0^{\infty}{ \phi(\Flow[(a.)][t, 0](x)) \H[x][(a.)](t, 0) \nu(dx)}
\]
Using Remarks~\ref{remark:inequality Ha} and \ref{remark:b0f} (\ref{remark:b0f:lim f flow a}), together with the fact that $\lambda' < f(\sigma_0)$, we deduce that 
\[ \forall t \geq 0,\quad 	\int_0^{\infty}{ \phi(\Flow[(a.)][t, 0](x)) \H[x][(a.)](t, 0) \nu(dx)} \leq D e^{-\lambda' t}  \]
for some constant $D$.
Moreover, one has, using  the change of variable $x = \Flow[a^*][v](0)$
\begin{align*}
\nu^\infty_{a^*}(\phi) & =
\int_0^{\sigma_{a^*}}{\phi(x) \nu^\infty_{a^*} (dx) } = \int_0^\infty{\phi(\Flow[a^*][v](0)) \gamma(a^*) \H[][a^*](v) dv } \\
&= \int_0^t{ \phi(\Flow[a^*][t,u](0)) \H[][a^*](t, u) \gamma(a^*) du } + \int_t^\infty{ \phi(\Flow[a^*][v](0)) \gamma(a^*) \H[][a^*](v) dv  }.
\end{align*}
The last equality is obtained with the change of variable $v = t - u$.
The second term is controlled by 
\begin{align*}
\int_t^\infty{ \phi(\Flow[a^*][u](0)) \gamma(a^*) \H[][a^*](u) du  } &\leq ||\phi||_\infty \gamma(a^*) \int_t^\infty{\frac{f(\Flow[a^*][u](0))}{\inf_{v \geq t }f(\Flow[a^*][v](0))} \EXP{-\int_0^u{f(\Flow[a^*][\theta](0)) d\theta}} du} \\
&= \frac{||\phi||_\infty \gamma(a^*)}{\inf_{v \geq t }f(\Flow[a^*][v](0))} \EXP{-\int_0^t{f(\Flow[a^*][\theta](0)) d\theta}} \\
& \leq D e^{-\lambda' t},
\end{align*}
for some constant $D$. We used again Remark~\ref{remark:b0f}. It remains to show that
\[  \Delta := \left| \int_0^t{\phi(\Flow[(a.)][t,u](0)) \H[][(a.)](t, u)  \r[\nu][(a.)](u, 0) du} -  \int_0^t{ \phi(\Flow[a^*][t,u](0)) \H[][a^*](t, u) \gamma(a^*) du } \right| \]
goes to zero exponentially fast.
One has
\begin{align*}
\Delta & \leq   \int_0^t{ \left|\phi(\Flow[(a.)][t, u](0)) -  \phi(\Flow[a^*][t, u](0))\right| \H[][(a.)](t, 
u) \r[\nu][(a.)](u, 0)  du } + \int_0^t{ \left|\H[][(a.)](t, u) - \H[][a^*](t, u)\right|   \phi(\Flow[a^*][t, 
u](0))   \r[\nu][(a.)](u, 0) du} \\
 &\quad  + \int_0^t{\H[][a^*](t, u) \phi(\Flow[a^*][t, u](0)) \left|\r[\nu][(a.)](u, 0) - \gamma(a^*)\right| du } \\
  &=:  \alpha_t + \beta_t + \theta_t.
\end{align*}
Using that for all $t \geq 0$, $|\r[\nu][(a.)](t, 0) - \gamma(a^*)| \leq C e^{-\lambda' t}$ ($\lambda' < \lambda$) and Remark~\ref{remark:inequality Ha}, we obtain:
\begin{align*}
\theta_t  &\leq C||\phi||_\infty  \int_0^t{ H_0(t, u) e^{-\lambda ' u} du} \\
 &= C ||\phi||_\infty   e^{-\lambda' t} \int_0^t{ H_0(t - u) e^{\lambda ' (t- u) } du} \\
 & \leq  \left[ C||\phi||_\infty  \int_0^\infty{H_0(u) e^{\lambda ' u } du} \right] e^{-\lambda' t} =: De^{-\lambda' t}.
\end{align*} 
The fact that $u \mapsto H_0(u) e^{ \lambda ' u}$ belongs to $L^1(\mathbb{R}_+)$ follows from $\lambda' < f(\sigma_0)$.
By Theorem~\ref{th:existence et unicité de l'équation limite}, one can find a constant $\bar{p}$ (with $\gamma(a^*) \leq \bar{p}$) such that:
\[ \forall t \geq 0,~ \E f(X_t) = \r[\nu][(a.)](t, 0) \leq \bar{p}. \]
Moreover, Assumption~\ref{assumptions:b0:Cvarphi} and Remark~\ref{remark:inequality Ha} give
\begin{align*}
\alpha_t &\leq \bar{p} l_\phi \int_0^t{|\Flow[(a.)][t,u](0) - \Flow[a^*][t,u](0)| H_0(t, u) du} \\
& \leq \bar{p} l_\phi C_{\varphi} \int_0^t{ \int_u^t{|a_\theta - a^*| d \theta} H_0(t, u) du}.
\end{align*}
Using that $\int_u^t{|a_\theta - a^*| d \theta}  
\leq J C \int_u^t{e^{-\lambda' \theta} d\theta} 
\leq \frac{J C e^{-\lambda' u}}{\lambda'}$, 
one has
\begin{align*}
\alpha_t &\leq \frac{\bar{p} l_\phi C_{\varphi} J C }{\lambda'} e^{-\lambda' t} \int_0^t{  e^{\lambda' (t-u)} H_0(t -u) du} \\
& \leq \left[\frac{\bar{p} l_\phi C_{\varphi} J C }{\lambda'} \int_0^\infty{H_0(u) e^{\lambda ' u } du}\right] e^{-\lambda' t} =: D e^{-\lambda' t}.
\end{align*}
Finally, using the inequality $|e^{-A} - e^{-B}| \leq e^{-\min(A, B)}|A- B|$ together with Remark~\ref{remark:inequality Ha}, we obtain
\[ \beta_t \leq ||\phi||_\infty \bar{p} \int_0^t{ \H[][0](t-u) \int_u^t{ \left|f(\Flow[(a.)][\theta, u](0)) - f(\Flow[a^*][\theta, u](0))\right| d \theta} du}. \]
Setting $\bar{a} := J \bar{p}$, we have moreover, using that $f'$ is non-decreasing and Lemma~\ref{lemma:b0:comparison principle}
\[   \int_u^t{ \left|f(\Flow[(a.)][\theta, u](0)) - f(\Flow[a^*][\theta, u](0)) \right| d \theta}  \leq f'(\Flow[\bar{a}][t, u]) \int_u^t{ \left| \Flow[(a.)][\theta, u](0) - \Flow[a^*][\theta, u](0) \right|  d \theta}. \]
Assumption~\ref{assumptions:b0} yields
\begin{align*}
 \int_u^t{ \left|f(\Flow[(a.)][\theta, u](0)) - f(\Flow[a^*][\theta, u](0)) \right| d \theta} & \leq C_\varphi f'(\Flow[\bar{a}][t, u](0)) \int_u^t{ \int_u^\theta{|a_s- a^*| ds} d \theta}  \\
& \leq C_\varphi J C f'(\Flow[\bar{a}][t, u](0)) \int_u^t{ \int_u^\theta{e^{-\lambda' s} ds} d \theta} \\
& \leq C_\varphi \frac{J C}{\lambda'} f'(\Flow[\bar{a}][t, u](0))  (t-u) e^{\lambda' (t-u)} e^{-\lambda' t}.
\end{align*}
We used the fact that
\[   \int_u^t{ \int_u^\theta{e^{-\lambda' s} ds} d \theta} =  \int_u^t{ \frac{e^{-\lambda' u} - e^{-\lambda' \theta}}{\lambda'} d \theta} \leq \frac{(t-u) e^{-\lambda' u}}{\lambda' }. \]
Note that Lemma~\ref{lemma:b0:drift bounded from above} implies that $f'(\Flow[\bar{a}][t, u](0)) \leq f'(C^{\bar{a}}_b (t-u))$ and using Remark~\ref{remark:b0f}(\ref{remark:b0f:f does not grow too much}) we have
\[ \forall \epsilon > 0, \exists A_\epsilon:~\forall x \geq 0, f'(x) \leq A_\epsilon e^{\epsilon x}. \]
Choosing $\epsilon := (f(\sigma_0) - \lambda') /2$, we obtain
\[  \int_u^t{ \left|f(\Flow[(a.)][\theta, u](0)) - f(\Flow[a^*][\theta, u](0))\right| d \theta}  \leq A_\epsilon C_\varphi \frac{J C}{\lambda'}  (t-u) e^{(\lambda' + \epsilon) (t-u)} e^{-\lambda' t},  \]
and we deduce that
\[ \beta_t \leq \left[ \frac{A_\epsilon J C_\varphi C  ||\phi||_\infty \bar{a}}{\lambda'} \int_0^{+\infty}{H_0(u) u e^{(\lambda' + \epsilon) u} du} \right] e^{-\lambda' t} =: D e^{-\lambda' t}. \]
Combining the three estimates, we have proved the result.
\end{proof}
\section{Long time behavior with constant drift}
\label{sec: long time behavior for isolated neuron}
\label{sec:comportement en temps long avec drift constant}

The goal of this section is to study the rate of convergence to the invariant probability measure when $J = 0$ (no interaction). We use Laplace transform techniques to characterize the convergence. We state here the main result of the section.
\begin{proposition}
	\label{prop: convergence du taux de saut dans le cas J=0}
	Grant Assumptions~\ref{assumptions:b0}, \ref{assumptions:f0} and \ref{assumptions:nu0}.
	Let $(\Y[][\nu][a][t])_{t \geq 0}$ be the solution of \eqref{non-homogeneous-SDE}, driven by a constant current $(a_t) \equiv a$, $a \geq 0$; starting at time $s = 0$ with law $\nu$. 
 One can find a constant $\lambda^*_a \in (0, f(\sigma_a)]$ (only depending on  $b$, $f$ and $a$) such that for any $0 < \lambda < \lambda^*_a$ it holds
\begin{equation} \label{eq:convergence_taux_de_saut_lineaire_vers_gamma}
\forall t \geq 0,\quad |\E f(\Y[][\nu][a][t]) - \gamma(a)| \leq D e^{-\lambda t} 
\int_0^\infty{ [1 + f(x)] |\nu - \nu^\infty_a| (dx)},
\end{equation}
where D is a constant only depending on $f, b,a$ and $\lambda$. Moreover, one has
\[ 
\Y[][\nu][a][t]\limhb{\longrightarrow}{t\to \infty}{\mathcal{L}}\nu^\infty_{a}. \]
\end{proposition}
\begin{remark}
	In the above result, 
	$\lambda^*_a$ is explicitly known in terms of $f, b$ 
	and $a$ (see its expression \eqref{eq:optimal rate of convergence lambda^*}) and is optimal (see Remark~\ref{rq:vitesse de convergence optimale}). Note also that \eqref{eq:convergence_taux_de_saut_lineaire_vers_gamma} states explicitly the dependence on the initial distribution $\nu$ through its distance to the invariant measure $\nu^\infty_{a}$.
\end{remark}
\subsection{Study of the Volterra equation}

In the case where $(a_t)$ is constant and equal to $a$, the Volterra equation \eqref{eq:equation de Volterra} is a linear homogeneous convolution Volterra equation. 
If moreover the initial condition $\nu$ is $\delta_0$, the kernel $\r[][a](t) := \E f(\Y[][\delta_0][a][t])$ satisfies \begin{equation}
\r[][a] = \K[][a] + \K[][a] * \r[][a],
\label{volterra_convo} 
\end{equation}
For such equations, it is very natural to use Laplace transform techniques as convolutions become scalar products with this transformation.
Furthermore, the ``kernel'' $\K[][a]$ and the ``forcing term'' $\K[\nu][a]$ are non-negative. Volterra equation with positive kernels have been studied in the context of Renewal theory.
The main reference on this question is a paper of Feller \cite{feller1941}. 
We refer to \cite[Th. 4]{feller1941} for this method.
However, in our case the rate of convergence is exponential. In order to achieve the optimal rate of convergence, we use general methods from the Volterra integral equation theory, and especially the so called ``Whole-line Palay-Wiener'' Theorem.

Along this section, we grant Assumptions~\ref{assumptions:b0}, \ref{assumptions:f0} and \ref{assumptions:nu0}. 

\begin{definition}[Laplace transform]
Let $g: \mathbb{R}_+ \mapsto \mathbb{R}$ be a measurable function.
The Laplace transform of $g$ is the following function
\[ 
\widehat{g}(z) := \int_0^\infty{ e^{-zt} g(t) dt, }
\]
defined for all $z \in \mathbb{C}$ for which the integral exists.
\label{def:transformée de laplace}
\end{definition}
Note that the Laplace transforms of $\H[][a]$ and $\K[][a]$ are well defined for all $z \in \mathbb{C}$ with $\RePart(z) > -f(\sigma_a)$. This follows from the fact that $\forall \lambda < f(\sigma_a),~  \sup_{t \geq 0}\H[][a](t)e^{\lambda t} < \infty$. 
The same holds for $\K[][a]$.
Integrating by parts the Laplace transform of $\K[][a]$ shows that
\begin{equation}
\forall z \in \mathbb{C}, ~\RePart(z) > -f(\sigma_a) \implies \widehat{\K[][a]}(z) = 1  - z \widehat{\H[][a]} (z).
\label{laplace_transform_k0}
\end{equation}
It is also useful to introduce the following Banach space
\begin{definition}
For any $\lambda \in \mathbb{R}$, let $L_{\lambda} = \{ f \in \mathcal{B}(\mathbb{R}_+, \mathbb{R}) : ||f||_{\lambda,1} < \infty \}$ the space of Borel-measurable functions from $\mathbb{R}_+$ to $\mathbb{R}$, equipped with the norm
 \[  ||f||_{\lambda,1} = \int_{\mathbb{R}_+} { |f(s)| e^{\lambda s} ds}.  \]
 \label{def:espace de banach a poids cas convo}
 \end{definition}
The long time behavior of $\r[][a]$ is related to the location of the poles of \(\widehat{r}_{a}\). Equation \eqref{volterra_convo} gives
\[ 
\forall \RePart(z) >0 \quad  \widehat{r}_a(z)  = \frac{ \widehat{K}_a(z) }{ 1 - \widehat{K}_a(z) }. 
\]
This suggests to study the location of the zeros of $1 - \widehat{K}_a(z)  = z \widehat{H}_a(z)$.

\subsection{On the zeros of \protect{$\widehat{H}_a$}}
\begin{lemma} $\forall z \in \mathbb{C},~\RePart(z) \geq 0 \implies  \widehat{H}_a(z) \neq 0.$
\label{lem: pas de zeros pour Re(z) >= 0}
\end{lemma}
\begin{proof}
	Remark first that \(\H[][a]\) being a real-valued function, $\widehat{H}_a(z) = 0$
	iff \(\widehat{H}_a(\bar{z}) = 0\), so it is sufficient to locate the zeros of $\widehat{H}_a$ in the region $\ImPart(z) \geq 0$.
	Next, it follows from for the non-negativity of $\K[][a]$ that
\[  |\widehat{K}_a(z)|  \leq \int_0^\infty{ |e^{-t z}| \K[][a](t) dt } <  \int_0^\infty{ \K[][a](t) dt } = 1  \text{ if } \RePart(z) >  0. \]
It yields $\RePart(z) >  0 \implies \widehat{H}_a(z) \neq 0$.
Moreover, following \cite{feller1941} proof of Theorem 4, (b), if $z = i y$, $y > 0$ then
\[
 i y \widehat{H}_a(i y) = 1 - \widehat{K}_a(i y) = \int_0^\infty{(1 - \cos( yt)) \K[][a](t) dt } + i  \int_0^\infty{\sin(y t) \K[][a](t) dt. }
 \]
Consequently, $\widehat{K}_a(i y)  = 1$ for some $y > 0$ would imply that for 
Lebesgue almost every $t \geq 0$, $(1 - \cos( yt)) \K[][a](t) = 0$, that is, a.e. $\K[][a](t) = 0$. It obviously contradicts the 
assumption \(f(x)>0\) for \(x>0\). It follows that 
$\forall y > 0,~ \widehat{H}_a(i y) \neq 0$. Finally for $z = 0$, we have $\widehat{H}_a(0) = \int_0^\infty{\H[][a](t) dt} \neq 0$.
\end{proof}

\begin{lemma}
The zeros of $\widehat{H}_a$ are isolated.
\label{lem: les zeros sont isoles} 
\end{lemma}
\begin{proof}
This directly follows from the fact that $\widehat{H}_a$ is an holomorphic function on $\RePart(z) > -f(\sigma_a)$ and thus its zeros are isolated.
\end{proof}

\begin{lemma}
For all $z \in \mathbb{C}$, it holds that
\[  |\widehat{K}_a(z)| \leq \frac{ \phi_a(\RePart(z) )}{|\ImPart(z)| },  \]
where for all $x \in \mathbb{R}$, $\phi_a(x) := ||K'_{a,x}||_1$ and $K_{a,x}(t) := e^{-x t} \K[][a](t)$,~ $K'_{a,x}(t) := \frac{d}{d t} K_{a,x}(t)$.

Consequently, the zeros of $\widehat{H}_a$ are within a ``cone'':
\[ \forall z \in \mathbb{C},~\RePart(z) > -f(\sigma_a),~z  = x + i y,~ \widehat{H}_a(z) = 0 \implies |y| \leq \phi_a(x). \]
\label{lem:les zeros sont dans un cone}
\end{lemma}
\begin{proof}
Let $z = x+iy$, $y > 0, x > -f(\sigma_a)$. We have
\begin{align*}
\widehat{K}_a(z) &= \int_0^\infty{ e^{-z t} \K[][a](t) dt}  = \int_0^\infty{ e^{-i y t} K_{a,x}(t) dt} = \int_0^\infty{ \frac{e^{-i y t}}{i y } K'_{a,x}(t) dt. }
\end{align*}
The last equality follows by an integration by part. It yields
\[ |\widehat{K}_a(z)| \leq \frac{ ||K'_{a,x}||_1}{|y| }. \]
We deduce that for $ |y| >  ||K'_{a,x}||_1$,  we have $\widehat{K}_a(z) \neq  1$ and also $\widehat{H}_a(z) \neq  0.$
\end{proof}
Consequently, from Lemmas~\ref{lem: pas de zeros pour Re(z) >= 0}, \ref{lem: les zeros sont isoles} and \ref{lem:les zeros sont dans un cone}, we can define the abscissa of the ``first'' zero of $\widehat{H}_a$: 
\begin{equation}
\lambda^*_a := - \sup\{ \RePart(z)|~ \RePart(z) > -f(\sigma_a),~\widehat{H}_a(z) =  0 \} ,
\label{eq:optimal rate of convergence lambda^*}
\end{equation}
with the convention that $\lambda^*_a= f(\sigma_a)$ if the set of zeros is empty. We have proved that
\[ 
0 < \lambda^*_a \leq f(\sigma_a) \leq \infty. 
\]
The parameter $\lambda^*_a$ is key here as it gives the speed of convergence to the invariant probability measure. It only
 depends on $a$, $b$ and $f$.
\subsection{Convergence with optimal rate}

Our goal in this section is to prove the following proposition
\begin{proposition}
Eq. \eqref{volterra_convo} has a unique solution $\r[][a]$ of the form: 
\[ \r[][a]= \gamma(a) + \xi_a \quad \text{ with } \quad \forall \lambda \in [0,\lambda^*_a), \quad ~~ \xi_a \in  L_{\lambda}. \]
The constant $\lambda^*_a > 0$ is defined by \eqref{eq:optimal rate of convergence lambda^*}.
\label{prop: comportement en temps long equation de volterra cas convolutionnel}
\end{proposition}

This result can be deduced from general theorems of the Volterra equations 
theory. For instance, one can apply 
\cite[Th. 2.4, Chap. 7]{gripenberg_volterra_1990}. 
However, this last result is written for general measure kernels in weighted 
spaces and its proof is somehow difficult to follow. In our setting, the proof 
given by \cite{gripenberg_volterra_1990} simplifies a lot and we give it here 
for completeness. We use the following so-called ``Whole Line Palay-Wiener'' 
Theorem which is one of the most important ingredients of the convolution 
Volterra integral equations theory.
\begin{theorem}[Whole-line Palay-Wiener]
Let $k \in L^1(\mathbb{R}, \mathbb{R})$.
There exists a function $x \in L^1(\mathbb{R}, \mathbb{R})$ satisfying the equation 
\[ \forall t \geq 0,~x(t) = k(t) + \int_{\mathbb{R}}{k(t-u) x(u) du}\]
if and only if
\[ \forall y \in \mathbb{R}, ~ \widehat{k}(i y) := \int_{\mathbb{R}}{e^{- i y t }  k(t) dt } \neq 1. \]
\label{th:whole line palay wiener theorem}
\end{theorem}
Note that here $\widehat{k}(i y)$ is actually the Fourier transform of $k$ evaluated at $y \in \mathbb{R}$.
\begin{proof}
See \cite[Th. 4.3, Chap. 2]{gripenberg_volterra_1990}. We prove later, in details, an extension of this theorem (see Proposition~\ref{prop:whole line palay wiener theorem extension}).
\end{proof}

\begin{proof}[Proof of Proposition~\ref{prop: comportement en temps long equation de volterra cas convolutionnel}]
Let $\sigma_-$ and $\sigma_+$ be any real numbers such that:
\[ -\lambda^*_a < \sigma_- < 0 < \sigma_+ < \infty. \] 
We first extend $\r[][a]$, $\K[][a]$ and $\H[][a]$  to the whole line by defining: $\forall t \in \mathbb{R},~\r[][a](t) := \r[][a](t) \indic{t \geq 0},~\K[][a](t) :=\K[][a](t) \indic{t \geq 0}$ and $\H[][a](t) := \H[][a](t) \indic{t \geq 0}$. We have from \eqref{volterra_convo}
\begin{equation}
 \forall t \in \mathbb{R},~ \r[][a](t) = \K[][a](t)  + \int_\mathbb{R}{\K[][a](t-u) \r[][a](u) du}. 
 \label{eq:r_a sigma + equation sur R}
 \end{equation}
For any $ \vartriangle \in \mathbb{R}$, we also define $r_{a,\vartriangle}(t):= e^{- \vartriangle t} \r[][a](t),~~K_{a,\vartriangle}(t) := e^{- \vartriangle t} \K[][a](t)$.
Note that $K_{a,\sigma_-} \in L^1(\mathbb{R})$ and that $\forall y \in \mathbb{R},~\widehat{K}_{a,\sigma_-}(i y) = \widehat{K}_a(\sigma_- + i y ) \neq 1$ (by definition of $\lambda^*_a$).  We can apply Theorem~\ref{th:whole line palay wiener theorem}: there exists $\xi_{a, \sigma_-} \in L^1(\mathbb{R})$ such that
\begin{equation}  \forall t  \in \mathbb{R},~ \xi_{a,\sigma_-}(t) = K_{a,\sigma_-}(t) + \int_{\mathbb{R}}{K_{a,\sigma_-}(t-u) \xi_{a,\sigma_-}(u) du }. 
\label{eq:xi a sigma - equation sur R}
\end{equation}
We define $\xi_a(t) := e^{\sigma_- t} \xi_{a,\sigma_-}(t)$. We have $\int_{\mathbb{R}}{|\xi_a(u)| e^{-\sigma_- u } du } < \infty$ and \eqref{eq:xi a sigma - equation sur R} reads
\[ \forall t \in \mathbb{R},~ \xi_a(t)  = \K[][a](t)  + \int_{\mathbb{R}} {\K[][a](t-u) \xi_a(u) du}. \]
\begin{remark}The function $\xi_a$ is not null on $\mathbb{R}_-$ (see formula \eqref{eq:complete link between ra and xia on R} just below).
\end{remark}
We have, using equalities 
\eqref{eq:r_a sigma + equation sur R} 
and \eqref{eq:xi a sigma - equation sur R} 
\begin{align*}
\xi_{a, \sigma_-} \in L^1(\mathbb{R}),~~\widehat{\xi_{a, \sigma_-}}(iy) = \left[ \frac{\widehat{K_a}}{1-\widehat{K_a}} \right](iy + \sigma_-), \\
r_{a, \sigma_+} \in L^1(\mathbb{R}),~~\widehat{r_{a, \sigma_+}}(iy) =  \left[ \frac{\widehat{K_a}}{1-\widehat{K_a}} \right](iy + \sigma_+). 
\end{align*}
We can now use the Fourier inverse formula for $L^1(\mathbb{R})$ functions to get
\[
 \xi_{a, \sigma_-}(t) = \frac{1}{2 \pi} \int_{\mathbb{R}} { e^{i y t }  \left[ \frac{\widehat{K_a}}{1-\widehat{K_a}} \right] (iy + \sigma_-) dy } \quad \text{ and } \quad  r_{a, \sigma_+}(t) = \frac{1}{2 \pi} \int_{\mathbb{R}} { e^{i y t }   \left[  \frac{\widehat{K_a}}{1-\widehat{K_a}} \right](iy + \sigma_+) dy },
\]
or after the changes of variable $z = iy + \sigma_-$ and $z = iy + \sigma_+$:
\[  \xi_a(t) = \lim_{T \rightarrow \infty}{\frac{1}{2 \pi i} \int_{\sigma_- - i T}^{\sigma_- + i T } { e^{z t} \frac{\widehat{K}_a(z)}{ 1 -  \widehat{K}_a(z)}  dz} } \quad \text{ and } \quad r_a(t) = \lim_{T \rightarrow \infty}{\frac{1}{2 \pi i} \int_{\sigma_+ - i T}^{\sigma_+  + i T } { e^{z t} \frac{\widehat{K}_a(z)}{ 1- \widehat{K}_a(z)}  dz} }. \]
Let $\Gamma_T$ be the closed curve in the complex plane composed of four straight lines that join the points $\sigma_- - i T$, $\sigma_- + iT$, $\sigma_+ + iT$, and $\sigma_+ - iT$ in the anti-clockwise direction. It follows from the residue theorem that
\begin{equation}
\int_{\Gamma_T}{ e^{z t}  \frac{\widehat{K}_a(z)}{ 1 - \widehat{K}_a(z)}  dz } = \int_{\Gamma_T}{ e^{z t}  \frac{\widehat{K}_a(z)}{ z \widehat{H}_a(z)}  dz } = 2 \pi i \frac{\widehat{K}_a(0)} { \widehat{H}_a(0) } = 2 \pi  i \gamma(a).
\label{eq:residu theorem sur le chemin GammaT}
\end{equation}
The last equality follows from
\[ \widehat{H}_a(0) = \int_0^\infty{ \H[][a](t) dt} = \int_0^\infty{ \EXP{-\int_0^t{f(\Flow[a][u]) du}} dt } \stackrel{(\ref{eq:equation gamma a avec le flot})}{=}   \frac{1}{\gamma(a)}.\]
By Lemma~\ref{lem:les zeros sont dans un cone}, for all $z$ in the strip $\RePart(z) \in [\sigma_-, \sigma_+]$, $z \neq 0$, we have
\[ |\widehat{K}_a(z)| \leq \frac{\phi_a (\sigma_-) }{ |\ImPart(z)| }. \]
We deduce that
\[ \lim_{T \rightarrow \pm \infty} \int_{\sigma_- + i T}^{\sigma_+ + i T } { e^{z t} \frac{\widehat{K}_a(z)}{ 1 -  \widehat{K}_a(z)}  dz}  = 0.\]
Therefore we can take the limit $T \rightarrow \infty$ in \eqref{eq:residu theorem sur le chemin GammaT} and obtain
\begin{equation} \forall t \in \mathbb{R},~ \r[][a](t) = \gamma(a) + \xi_a(t). 
\label{eq:complete link between ra and xia on R}
\end{equation}
The proposition is proven by choosing $\sigma_- = -\lambda$.
\end{proof} 

\begin{remark}
The speed of convergence obtained in this result is optimal if $\lambda^*_a < f(\sigma_a)$ (i.e. $\widehat{H}_a$ has at least one complex zero with $\RePart(z) > -f(\sigma_a)$) in the sense that
 \[ \forall \lambda > \lambda^*_a,~  r_a - \gamma(a) \notin L_\lambda. \]
To see this, assume that $\lambda^*_a  < f(\sigma_a)$ and choose $\sigma_-$ such that $-f(\sigma_a) < \sigma_- < -\lambda^*_a$. The previous proof can be mimicked except that the  residues of equation \eqref{eq:residu theorem sur le chemin GammaT}  now involves terms of the order $e^{-\lambda^*_a t}$ -  corresponding to the roots of $\widehat{H}_a$ with real part equal to $-\lambda^*_a$. 
\label{rq:vitesse de convergence optimale}
\end{remark}
\subsection{Long time behavior starting from initial condition $\nu$: proof of Proposition~\ref{prop: convergence du taux de saut dans le cas J=0}}

We now come back to the general case where the initial condition can be any probability measure satisfying Assumption~\ref{assumptions:nu0}, and we give the proof of Proposition~\ref{prop: convergence du taux de saut dans le cas J=0}. 
\begin{proof}[Proof of Proposition~\ref{prop: convergence du taux de saut dans le cas J=0}]
	Note  that, we only consider here the convolutions on \([0,t]\) denoted by \(*\) (and no more the convolution on \(\mathbb{R}\)).
Let $\r[\nu][a](t) = \E f(\Y[][\nu][a][t])$ with $\mathcal{L}(Y_0) = \nu$. The function $\r[\nu][a]$ is the unique solution of the Volterra equation
\[ \r[\nu][a] =\K[\nu][a] + \K[][a] * \r[\nu][a]. \]
If we choose $\nu$ to be the invariant probability measure $\nu^\infty_a$, we get $\gamma(a) = \K[\nu^\infty_a][a] + \K[][a] * \gamma(a)$ and
\[\r[\nu][a] - \gamma(a) = \K[\nu][a] - \K[\nu^\infty_a][a] + \K[][a] * (\r[\nu][a] - \gamma(a) ). \]
We can solve this equation in terms of $\r[][a]$, the ``resolvent'' of $\K[][a]$ (using 
Lemma~\ref{lem:solution a laide de la resolvante.})
 and obtain
\begin{align*}
\r[\nu][a] - \gamma(a) & = \K[\nu][a] - \K[\nu^\infty_a][a] + \r[][a] * (\K[\nu][a] - \K[\nu^\infty_a][a]) \\
&= \K[\nu][a] - \K[\nu^\infty_a][a] + \xi_a *  (\K[\nu][a] - \K[\nu^\infty_a][a]) + \gamma(a) * (\K[\nu][a] - \K[\nu^\infty_a][a]),
\end{align*}
where $\r[][a] = \xi_a  + \gamma(a)$, see \eqref{eq:complete link between ra and xia on R}, is the solution of the Volterra equation $\r[][a] = \K[][a] + \K[][a] * \r[][a]$. Using \eqref{relation entre K et H}, we have $\gamma(a)* \K[\nu][a]  = \gamma(a)(1- \H[\nu][a])$ and thus
\[ 
\r[\nu][a] - \gamma(a) = \K[\nu][a] - \K[\nu^\infty_a][a] + \gamma(a)(\H[\nu^\infty_a][a] - \H[\nu][a]) +   \xi_a *  (\K[\nu][a] - \K[\nu^\infty_a][a]).
\]
We now write $\Theta$ any constant only depending on $\lambda, f, b$ and $a$ and which may change from line to line.
It is clear that for any $0 < \lambda < f(\sigma_a)$
\[  |\H[\nu^\infty_a][a] - \H[\nu][a]|(t) \leq \int_0^\infty{\H[x][a](t) |\nu - \nu^\infty_a|(dx) } \leq \int_0^\infty{ \H[][a](t) |\nu - \nu^\infty_a|(dx) } \leq \Theta e^{-\lambda t} \int_0^\infty{ |\nu - \nu^\infty_a|(dx) }.  \]
Similarly, for any $0 < \lambda < f(\sigma_a)$,
\begin{align*}
 |\K[\nu][a] - \K[\nu^\infty_a][a]|(t) & \leq \int_0^\infty{ f(\Flow[a][t](x))\H[x][a](t) |\nu - \nu^\infty_a|(dx) } \leq  \int_0^\infty{ f(x + C^a_b t) \H[][a](t) |\nu - \nu^\infty_a|(dx)} \\
 & \leq C_f \int_0^\infty{ [1 + f(x) + f(C^a_b t)] \H[][a](t) |\nu - \nu^\infty_a|(dx)} \leq \Theta e^{-\lambda t} \int_0^\infty{(1+f(x))|\nu - \nu^\infty_a|(dx)}.
\end{align*} 
We used here Assumption~\ref{assumptions:f0:f(x+y)}.
Let now $0 < \lambda < \lambda^*_a$. Using \(\xi_a\in L_\lambda\), it holds that
\[ |\xi_a * (\K[\nu][a] - \K[\nu^\infty_a][a])|(t) \leq \int_0^t{ |\xi_a(t-u)| |\K[\nu][a] - \K[\nu^\infty_a][a]|(u) du } \leq \Theta e^{-\lambda t}\int_0^\infty{(1+f(x))|\nu - \nu^\infty_a|(dx)}. \]
Combining the three estimates, one deduces that
\[ |\r[\nu][a](t) - \gamma(a)| \leq \Theta e^{-\lambda t} \int_0^\infty{(1+f(x))|\nu - \nu^\infty_a|(dx)}. \]
It remains to prove that \(\lim_{t\rightarrow\infty} \mathcal{L}(\Y[][\nu][a][t]) =\nu^\infty_{a}\). 
The process $(\Y[][\nu][a][t])_{t \geq 0}$ is the solution of \eqref{NL-equation0} with $\tilde{b}(x) 
= b(x) + a$ and \(J=0\). Obviously, \(0\) solves \eqref{eq:scalar equation invariant measures}.
Applying Proposition~\ref{prop:convergence en loi vers la mesure invariante}
ends the proof.
\end{proof}
\section{Long time behavior with a general drift}
\label{sec: long time behavior for a general drift}
\label{section:methode de perturbation general drift}
In this section, we generalize the results obtained in Section~\ref{sec:comportement en temps long avec drift constant} to non constant currents. We consider the process \eqref{non-homogeneous-SDE} driven by a current $(a_t)$ assuming to converge exponentially fast to $a$. We seek to prove that the jump rate of this process is converging to $\gamma(a)$ and estimate the speed of convergence. 
This ``perturbation'' analysis will be useful to study the long time behavior of the solution of the non-linear McKean-Vlasov equation \eqref{NL-equation0} with small interactions. We consider a non-negative continuous function $(a_t)_{t \geq 0}$ such that
\begin{assumption}
\label{assumption:condition sur a_t}
\begin{enumerate}
\item $\sup_{t \geq 0} a_t \leq \bar{a}$ for some constant $\bar{a} > 0$.
\item There exist \(a \geq 0\), \(C\geq 0\) and 
\(\lambda \in(0, \min(\lambda^*_a, f(\sigma_0)))\), where \(\sigma_0\) and 
\(\lambda^*_a\) are defined by \eqref{eq:definition de sigma_a} and 
\eqref{eq:optimal rate of convergence lambda^*},
such that
	\begin{equation}
\forall t \geq 0,\quad |a_t - a| \leq C e^{-\lambda t}.
	\label{speed_of_convergence_of_a_t}
	\end{equation}
\end{enumerate}
\end{assumption}	
	Note that the values of $C$ and $\lambda$ are important in this analysis. 
	Any mention of $C$ and $\lambda$ in this section  refer to these two 
	constants.

Let $\r[\nu][(a.)](t, s) = \E f(\Y[s][\nu][(a.)][t])$, where $\Y[s][\nu][(a.)][t]$ is the solution of \eqref{non-homogeneous-SDE} driven by the current $(a_t)$ and starting at time $s$ with law $\nu$.
The goal of this section is to prove that if $C$ is small enough, then there exists an explicit constant $D$ such that 
     \[  
     \forall t \geq s \geq 0,~ |\r[\nu][(a.)](t, s) - \gamma(a)| \leq D e^{-\lambda (t-s)} ~,
     \]
     where $\gamma(a)$ is given by \eqref{eq:gamma(a) la constante de renormalisation}.
    Note that the exponential decay rate $\lambda$ is preserved.
    We make efforts to keep track of the constant $D$ and to relate it to $C$.
	As in Section~\ref{sec:comportement en temps long avec drift constant} it is useful to split the study in two parts: the case where the initial condition is a Dirac mass at $0$ and the general case.
	We thus consider the unique solution $\r[][(a.)]$ of the following Volterra equation:
	\begin{equation}
	\r[][(a.)] = \K[][(a.)] + \K[][(a.)]*\r[][(a.)].
	\label{eq:volterra equation, resolvent}
	\end{equation}
	It is also useful to introduce a Banach space adapted to this non-homogeneous setting.
\subsection{An adapted Banach algebra}

\begin{definition}
A function $K:~(\mathbb{R}_+)^2 \rightarrow \mathbb{R}$ is said to be a Volterra Kernel with weight $\lambda \in \mathbb{R}$ if: $K$ is Borel measurable, $\forall s > t: K(t, s) = 0 $ a.e. and $||K||_{\lambda,1} < \infty$ with
\[ ||K||_{\lambda,1} := \esssup_{t \geq 0}{\int_{\mathbb{R}_+}{|K(t, s)| e^{\lambda (t-s)} ds} }.  \]
We define $\mathcal{V}_{\lambda}$ the set of Volterra kernels with weight $\lambda$.
We also define for $K \in \mathcal{V}_\lambda$:
\[ ||K||_{\lambda, \infty} = \esssup_{t, s \geq 0}{|K(t,s)e^{\lambda (t-s)}|} \in \mathbb{R}_+ \cup \{+\infty \}. \]
\end{definition}
\begin{proposition}\label{lem:20190418}
	The space $(\mathcal{V}_{\lambda}, ||\cdot||_{\lambda,1})$ is a Banach algebra. Furthermore, 
for all $a,b \in \mathcal{V}_{\lambda}$, $||a*b||_{\lambda,1} \leq ||a||_{\lambda,1}||b||_{\lambda,1}$. 
\end{proposition}
Proposition~\ref{lem:20190418} 
is proved in \cite{gripenberg_volterra_1990}, Theorem 2.4 and Proposition 2.7 
(i) of Chapter 9.

\begin{lemma}[Connection with the time homogeneous setting]
Let $g \in L_{\lambda}$. We define
\[ 
\forall t,s \in \mathbb{R}_+, \quad \tilde{g}(t, s) := g(t-s)\indica{t \geq s}.
\]
Then $\tilde{g} \in \mathcal{V}_{\lambda}$ and $||g||_{\lambda,1} = ||\tilde{g}||_{\lambda,1}.$
\end{lemma}

This result allows us to consider elements of  $L_{\lambda}$ as elements of 
$\mathcal{V}_{\lambda}$. 
\textbf{Note that the algebra $L_{\lambda}$ is commutative 
	whereas $\mathcal{V}_{\lambda}$ is not.}

\subsection{The perturbation method}

Define $\BK[][(a.)] := \K[][(a.)] - \K[][a]$ and 
$\BH[][(a.)] := \H[][(a.)] - \H[][a]$.
\begin{lemma}
Grant Assumptions~\ref{assumptions:b0}, \ref{assumptions:f0}, \ref{assumptions:nu0} and \ref{assumption:condition sur a_t}. Then, there exists a continuous non-negative and integrable function $\eta$
 such that for all $t \geq s \geq 0$, one has
\begin{align*}
|\BK[][(a.)] (t, s) | \leq C  e^{-\lambda t} \eta(t-s), \\
|\BH[][(a.)](t, s) | \leq C  e^{-\lambda t} \eta(t-s).
\end{align*}
The function $\eta$ only depends on $b, \bar{a}, f$ and  $\lambda$ (in particular it does not depend on $C$). Furthermore, we can choose $\eta$ such that $||\eta||_1$ is a non-decreasing function of $\bar{a}$.
\label{lem: la fonction eta qui controle K bar(a) et H bar(a)}
\end{lemma}
\begin{proof}
Here, to simplify the notation, we write $\Flow[(a.)][t, s]$ for $\Flow[(a.)][t, s](0)$. We have
\begin{align*}
\BK[][(a.)] (t, s) & =  f(\Flow[(a.)][t, s])\EXP{-\int_s^t{f(\Flow[(a.)][u, s]) du}} - f(\Flow[a][t, s])\EXP{-\int_s^t{f(\Flow[a][u, s]) du}}\\
|\BK[][(a.)] (t, s)| & \leq  |f(\Flow[(a.)][t, s]) - f(\Flow[a][t,s])| \EXP{-\int_s^t{ f(\Flow[(a.)][u, s]) du}} \\
& \quad + f(\Flow[a][t, s]) \left|\EXP{-\int_s^{t}{f(\Flow[(a.)][u, s]) du}} - \EXP{-\int_s^t{f(\Flow[a][u, s]) du}}\right| \\
                &  =:   M_1 + M_2.
\end{align*}
Assumptions~\ref{assumptions:b0}, \ref{assumptions:f0:f(x+y)}
and \eqref{speed_of_convergence_of_a_t}
give
\begin{align*}
 |f(\Flow[(a.)][t, s]) - f(\Flow[a][t,s])|  & \leq  f'(\Flow[\bar{a}][t, s]) | \Flow[(a.)][t, s] - \Flow[a][t, s]| 
  \leq f'(C^{\bar{a}}_{b} (t-s)) C_\varphi  \int_s^t{|a_u - a| du} \\
  & \leq f'(C^{\bar{a}}_{b} (t-s)) C_\varphi C   \int_s^t{e^{-\lambda u} du} 
   \leq C e^{-\lambda t}  f'(C^{\bar{a}}_{b} (t-s)) C_\varphi   \frac{e^{\lambda(t-s)}}{\lambda}.
\end{align*}
Moreover choosing $\lambda'  \in (\lambda , f(\sigma_0))$ and using the fact that $f(\Flow[0][u]) \rightarrow f(\sigma_0)$ as $u \rightarrow \infty$, one obtains
\begin{align*}
 \EXP{-\int_s^t{ f(\Flow[(a.)][u, s]) du}} &\leq \EXP{-\int_s^t{ f(\Flow[0][u, s]) du}} =  \EXP{-\int_0^{t-s}{ f(\Flow[0][u]) du}} \\
 & \leq D(b, f, \lambda') e^{-\lambda' (t-s)},
\end{align*}
for some finite constant $D(b, f, \lambda')$. Let $\alpha(u) :=\frac{D(b, f, \lambda')}{\lambda} e^{-(\lambda' - \lambda) u} f'(C^{\bar{a}}_{b} u) C_\varphi$, we have
\[ 
M_1 \leq C e^{-\lambda t}  \alpha(t-s), 
\]
and $\alpha \in L^1(\mathbb{R}_+)$.
Moreover, for \(A, B \geq 0\), we have $|e^{-A} - e^{-B}| \leq e^{-\min(A, B)}|A-B|$. So,
\begin{align*}
M_2 &\leq f(\Flow[\bar{a}][t, s]) 
\EXP{-\int_0^{t-s}{f(\Flow[0][u])du}} 
\left|\int_s^t{ f(\Flow[(a.)][u,s]) - f(\Flow[a][u, s]) du}\right| \\
& \leq f(C^{\bar{a}}_b (t-s)) D(b, f, \lambda') e^{-\lambda' (t-s)}   f'( C^{\bar{a}}_b (t-s))   \int_s^t{ |\Flow[(a.)][u,s] - \Flow[a][u, s] |du}. 
\end{align*} One has
\[
\int_s^t{ |\Flow[(a.)][u,s] - \Flow[a][u, s] |du} \leq C_\varphi \int_s^t{\int_s^u{|a_\theta - a| d\theta} du} \leq C C_\varphi \int_s^t{\int_s^u{e^{-\lambda \theta} d\theta} du} 
\leq C e^{-\lambda t} \cdot \frac{C_\varphi}{\lambda}(t-s)e^{\lambda (t-s)}. 
\]
Consequently $M_2 \leq C e^{-\lambda t}  \beta(t-s)$ with
\[ 
\beta(u) :=D(b, f, \lambda') e^{-(\lambda' - \lambda)u} f(C^{\bar{a}}_b u)  f'( C^{\bar{a}}_b u)  \frac{C_\varphi}{\lambda}u e^{\lambda u} .  
\]
It holds that $\beta \in L^1(\mathbb{R}_+)$ and setting $\eta := \alpha + \beta$ completes the proof for $\BK[][(a.)]$. The same computations give a similar result for $\BH[][(a.)] $.
\end{proof}
These estimates are sharp enough to give the following result:
\begin{lemma} 
Grant Assumptions~\ref{assumptions:b0}, \ref{assumptions:f0}, \ref{assumptions:nu0} and \ref{assumption:condition sur a_t}. Let $\eta$ be the function given by Lemma~\ref{lem: la fonction eta qui controle K bar(a) et H bar(a)}.
Denote by $1$ the kernel $\indica{t \geq s}$. Then
\begin{enumerate}
\item $\BK[][(a.)] \in \mathcal{V}_\lambda$ and $||\BK[][(a.)]||_{\lambda,1} \leq C ||\eta||_1$.
\item $\BK[][(a.)] * 1 \in \mathcal{V}_\lambda$ and $||\BK[][(a.)] * 1||_{\lambda,1} \leq C ||\eta||_1$.
\end{enumerate}
The exact same estimates holds for $\BH[][(a.)]$ and $\BH[][(a.)] * 1$.
\label{lem:K_bar * 1 et H_bar * 1}
\end{lemma}
\begin{proof}
Using Lemma~\ref{lem: la fonction eta qui controle K bar(a) et H bar(a)}, we have
\[ 
||\BK[][(a.)]||_{\lambda,1} := \sup_{t \geq 0}{\int_0^t{|\BK[][(a.)]|(t, s) e^{\lambda (t-s) }ds}}  \leq \sup_{t \geq 0} {\int_0^t{Ce^{-\lambda s} \eta(t-s) ds}} \leq C ||\eta||_1, 
\]
proving point 1. For point 2, we have $\forall t \geq s \geq 0$, $(\BK[][(a.)] * 1)(t, s) := \int_s^t{\BK[][(a.)] (t, u) du}$. 
And Lemma~\ref{lem: la fonction eta qui controle K bar(a) et H bar(a)} gives
\[ 
||\BK[][(a.)] * 1||_{\lambda,1} = \sup_{t \geq 0}{\int_0^t{|\BK[][(a.)] * 1|(t, s) e^{\lambda (t-s) }ds}} \leq \sup_{t \geq 0}{\int_0^t{C e^{-\lambda t} ||\eta||_1   e^{\lambda (t-s) } ds}} = C ||\eta||_1. 
\]
\end{proof}
\begin{proposition}
Grant Assumptions~\ref{assumptions:b0}, \ref{assumptions:f0}, \ref{assumptions:nu0}. Assume $(a_t)_{t \geq 0}$ satisfies Assumption~\ref{assumption:condition sur a_t} and that the constant $C$ is small enough:
\begin{equation}   
	\alpha := C ||\eta||_1 (1 + ||\xi_a||_{\lambda,1} + \gamma(a) ) <  1.
 \label{eq:contrainte satisfaite par C pour que la methode de perturbation s'applique}
\end{equation}
Define $\Delta_K := \BK[][(a.)] + \xi_a * \BK[][(a.)] - \gamma(a) \BH[][(a.)]$ and let $\Delta_r$ be the solution of the Volterra equation
\begin{equation} 
\label{eq:resolvante de l'equation de perturbation}
\Delta_r = \Delta_K + \Delta_K * \Delta_r. 
\end{equation}
Then
\begin{enumerate}
\item $\Delta_K \in \mathcal{V}_\lambda$ with $||\Delta_K||_{\lambda,1} \leq \alpha$ and $\Delta_K * 1 \in \mathcal{V}_\lambda$ with $||\Delta_K * 1||_{\lambda,1} \leq \alpha$.
\item $\Delta_r \in \mathcal{V}_\lambda$ with $||\Delta_r||_{\lambda,1} \leq \frac{\alpha}{1-\alpha}$ and $\Delta_r * 1 \in \mathcal{V}_\lambda$ with $||\Delta_r * 1||_{\lambda,1} \leq \frac{\alpha}{1-\alpha}$.
\item Consider $\r[][(a.)](t, s)$  the jump rate associated to the current $(a_t)_{t \geq 0}$. Then
 \begin{equation} \r[][(a.)] = \r[][a] + \Delta_r + \Delta_r * \r[][a]. 
\label{eq:solution of the perturbed resolvent equation using tilde(r)}
\end{equation}
\end{enumerate}
Consequently, we have $\r[][(a.)] = \gamma(a) + \xi_{(a.)}$ with
\[ \xi_{(a.)} = \xi_a + \Delta_r + \Delta_r * \xi_a + \gamma(a) (\Delta_r * 1) \in \mathcal{V}_\lambda. \]
Furthermore, 
\[ ||\xi_{(a.)} ||_{\lambda,1} \leq ||\xi_a||_{\lambda,1} + \frac{\alpha}{1-\alpha} [1 + ||\xi_a||_{\lambda,1} + \gamma(a)]. \]
\label{prop:long time behavior of the perturbed resolvant equation}
\end{proposition}

\begin{proof}
By Lemma~\ref{lem:K_bar * 1 et H_bar * 1}, we have $||\Delta_K||_{\lambda,1} \leq \alpha < 1$. Consequently equation \eqref{eq:resolvante de l'equation de perturbation} admits a unique solution $\Delta_r\in \mathcal{V}_\lambda$ satisfying
\(||\Delta_r||_{\lambda,1} \leq \frac{\alpha}{1-\alpha}. 
\)
The kernel $\Delta_r * 1$ satisfies the following Volterra equation
\begin{equation} \Delta_r * 1 = (\Delta_K * 1) + \Delta_K * (\Delta_r * 1)
\label{eq:volterra equation tilde(r) * 1 }
\end{equation}
with $ \Delta_K * 1 = (\BK[][(a.)] * 1) + \xi_a * (\BK[][(a.)] * 1) + \gamma(a) (\BH[][(a.)] * 1)$. It follows from Lemma~\ref{lem:K_bar * 1 et H_bar * 1} that $\Delta_K * 1 \in \mathcal{V}_\lambda$ and $||\Delta_K * 1||_{\lambda,1} \leq \alpha$.
From $||\Delta_K||_{\lambda,1} < 1$, one gets that equation \eqref{eq:volterra equation tilde(r) * 1 } has its solution in $\mathcal{V}_\lambda$ and
\[ \Delta_r* 1 \in \mathcal{V}_\lambda,~~ ||\Delta_r * 1||_{\lambda,1} \leq \frac{\alpha}{1-\alpha}. \]
It remains to check that 
\(\r[][(a.)]\)
given by \eqref{eq:solution of the perturbed resolvent equation using tilde(r)} 
is indeed the solution of \eqref{eq:volterra equation, resolvent}. Let $r := \r[][a] + \Delta_r + \Delta_r * \r[][a]$. 
One has
\begin{align*}
\Delta_K * r & = \Delta_K * \r[][a]  + (\Delta_r - \Delta_K) +  (\Delta_r - \Delta_K)  * \r[][a] \\
  & = \Delta_r * \r[][a]  + \Delta_r - \Delta_K \\
  &= r - \r[][a] - \Delta_K,
\end{align*}
i.e. $r$ satisfies 
\begin{equation}
\label{eq:perturbation_equation}
 r = \r[][a] + \Delta_K + \Delta_K * r.
\end{equation}
Using Proposition~\ref{prop: comportement en temps long equation de volterra cas convolutionnel} and \eqref{relation entre K et H}, we have
$\Delta_K = \BK[][(a.)] + \r[][a] * \BK[][(a.)]$. Equation \eqref{eq:perturbation_equation} gives
\[ r - (\BK[][(a.)] + \r[][a] * \BK[][(a.)]) * r = r_a  + \BK[][(a.)] + \r[][a] * \BK[][(a.)].  \]
We multiply this equation by $\K[][a]$ on the left and obtain, using that $\K[][a] * \r[][a] =  \r[][a] * \K[][a] = \r[][a] - \K[][a]$:
\[ \K[][a] * r - \r[][a] * \BK[][(a.)] * r = \r[][a] - \K[][a] + \r[][a] * \BK[][(a.)]. \]
The relation $\BK[][(a.)] = \K[][(a.)] - \K[][a]$ yields 
\[ \K[][a] * r - \r[][a] * \BK[][(a.)] * r = \r[][a] * \K[][(a.)], \]
or equivalently
\[ \Delta_K * r = \K[][(a.)] * r - \r[][a] * \K[][(a.)].  \]
We substitute this equality in \eqref{eq:perturbation_equation} and finally obtain
\[ \r = \K[][(a.)] + \K[][(a.)] * r. \]
By uniqueness (Lemma~\ref{lem:existence et unicté de l'équation de Volterra} with $\nu = \delta_0$), it follows that $\r = \r[][(a.)]$. The end of the proof follows easily.
\end{proof}
\begin{remark}
Let us explain how the formula \eqref{eq:solution of the perturbed resolvent equation using tilde(r)} 
was derived. The algebra $\mathcal{V}_\lambda$ does not have any neutral element (in fact the 
neutral element would be a Dirac distribution) but assume for the sake of this heuristic that $I$ is a 
neutral element of the algebra (i.e. $k * I = I * k = k ~\forall k \in \mathcal{V}_\lambda$). Equation  
\eqref{eq:volterra equation, resolvent} can be rewritten as
\begin{equation}
(I - \K[][(a.)]) * (I + \r[][(a.)]) = I.
\label{eq_r_matricielle}
\end{equation}
In particular (taking $(a_t) \equiv a$), we have $(I - \K[][a]) * (I + \r[][a]) = (I + \r[][a]) *  (I - \K[][a]) = I$. Furthermore,
\[ I - \K[][(a.)] = (I - \K[][a]) * (I - (I  + \r[][a]) *  \BK[][(a.)]) ,\]
with $\BK[][(a.)] =  \K[][(a.)] - \K[][a] \in \mathcal{V}_\lambda$. Equation \eqref{eq_r_matricielle} becomes $ (I - \K[][a]) * (I - (I  + \r[][a]) *\BK[][(a.)]) * (I + \r[][(a.)]) = I$. We multiply by $I + \r[][a]$ on the left of each side, and we get $ (I - (I  + \r[][a]) * \BK[][(a.)]) * (I + \r[][(a.)]) = I + \r[][a]$.

We now expand this equation - the neutral element $I$ disappears and obtain:
\[ \r[][(a.)] - (\BK[][(a.)] + \r[][a] * \BK[][(a.)]) * \r[][(a.)] = \r[][a] + \BK[][(a.)] + \r[][a] * \BK[][(a.)]. \]
Using the definition of $\Delta_K$ we obtain  $ \r[][(a.)] = \r[][a]  + \Delta_K *  \r[][(a.)] 
+\Delta_K.$ Solving this equation in terms of $\Delta_r$ the resolvent of $\Delta_K$ we have 
$\r[][(a.)] = \r[][a] + \Delta_K + \Delta_r* (\r[][a] + \Delta_K)$.  It gives the desired formula.
\end{remark}
We now come back to an arbitrary initial condition $\nu$ and prove the main result of this section.

\begin{proposition}
Grant Assumptions~\ref{assumptions:b0}, \ref{assumptions:f0} and 
\ref{assumptions:nu0}. Let $(\Y[s][\nu][(a.)][t])_{t \geq s}$ be the solution 
to the non-homogeneous equation \eqref{non-homogeneous-SDE} driven by current 
$(a_t)_{t \geq 0}$ and with distribution $\nu$ at time $s$.  Let 
$\r[\nu][(a.)](t, s) =\E f(\Y[s][\nu][(a.)][t])$.
Assume $(a_t)$ satisfies Assumption~\ref{assumption:condition sur a_t} and that the constant $C$  satisfies the inequality \eqref{eq:contrainte satisfaite par C pour que la methode de perturbation s'applique} for some $\alpha \in (0,1)$. Then it holds that
\[ \forall t \geq s \geq 0,~ |\r[\nu][(a.)](t, s) - \gamma(a)| \leq D e^{-\lambda (t-s)}, \]
with
\[  D :=  \frac{1+\alpha\gamma(a)+ ||\xi_a||_{\lambda,1}}{1-\alpha} ||\K[\nu][(a.)]||_{\lambda,\infty} + \gamma(a) ||\H[\nu][(a.)]||_{\lambda,\infty}. \]
\label{prop: theoreme general de perturbation avec les constantes explicites}
\end{proposition} 
\begin{proof}
The kernel $\r[\nu][(a.)]$ solves the Volterra equation $\r[\nu][(a.)] = \K[\nu][(a.)] + \K[][(a.)] * \r[\nu][(a.)]$. By Lemma~\ref{lem:solution a laide de la resolvante.}, its solution is 
\[ 
\r[\nu][(a.)] = \K[\nu][(a.)] + \r[][(a.)] * \K[\nu][(a.)]. 
\]
Using Proposition~\ref{prop:long time behavior of the perturbed resolvant equation}, we know that $\r[][(a.)] = \gamma(a) + \xi_{(a.)}$, with $ \xi_{(a.)} \in \mathcal{V}_\lambda$.
Furthermore using that $\gamma(a) * \K[\nu][(a.)] = \gamma(a) [1 - \H[\nu][(a.)]]$, we deduce that: 
\[\r[\nu][(a.)] = \gamma(a) + \K[\nu][(a.)] + \xi_{(a.)} * \K[\nu][(a.)]  - \gamma(a) \H[\nu][(a.)]. \]
Using that $\lambda < f(\sigma_0)$ (Assumption~\ref{assumption:condition sur a_t}) we find
\[ ||\H[\nu][(a.)]||_{\lambda,\infty} = \sup_{t,s} \H[\nu][(a.)](t, s) e^{\lambda (t-s)} < \infty,~ ||\K[\nu][(a.)]||_{\lambda,\infty} = \sup_{t,s} \K[\nu][(a.)](t, s) e^{\lambda (t-s)} < \infty. \]
We obtain
\begin{align*}
\forall t \geq s, ~|\r[\nu][(a.)](t, s) - \gamma(a)|e^{\lambda(t-s)} &\leq ||\K[\nu][(a.)]||_{\lambda,\infty} + \gamma(a) ||\H[\nu][(a.)]||_{\lambda,\infty} + e^{\lambda(t-s)} \int_s^t{ |\xi_{(a.)} | (t, u) \K[\nu][(a.)](u, s) du} \\
 &\leq ||\K[\nu][(a.)]||_{\lambda,\infty} + \gamma(a) ||\H[\nu][(a.)]||_{\lambda,\infty} + ||\K[\nu][(a.)]||_{\lambda,\infty}  \int_s^t{ |\xi_{(a.)} | (t, u)  e^{\lambda(t-u)} du} \\
 & \leq ||\K[\nu][(a.)]||_{\lambda,\infty}+ \gamma(a) ||\H[\nu][(a.)]||_{\lambda,\infty} + ||\K[\nu][(a.)]||_{\lambda,\infty}  ||\xi_{(a.)}||_{\lambda,1}. 
\end{align*}
Using the estimate of $||\xi_{(a.)}||_{\lambda,1}$ given by Proposition~\ref{prop:long time behavior of the perturbed resolvant equation}, we deduce the result.
\end{proof}

\section{Long time behavior for small interactions: proof of Theorem\protect~\ref{th:Jpetit}}
\label{sec:proof of the main theorem}
\subsection{Some uniform estimates}

We now turn to the proof of Theorem~\ref{th:Jpetit}. It is convenient to first extend the results obtained in Section~\ref{sec:comportement en temps long avec drift constant}: we need uniform estimates in the input current $a$.
In this section, we grant Assumptions~\ref{assumptions:b0}, \ref{assumptions:f0} and \ref{assumptions:nu0}.
\begin{lemma}
Let $\bar{a} > 0$. It holds that
\[ \inf_{a \in [0, \bar{a}]}{\lambda^*_a} > 0.\]
\label{lem: le min des lambda star}
\end{lemma}
\begin{proof}
We define the function $g$ related to the first zero of \(\widehat{H}_a\) 
by
\[
\forall a\in [0,\bar{a}],\quad g(a) := - \sup\{ \RePart(z)|~ \widehat{H}_a(z) =  0,~\RePart(z) > -f(\sigma_0) \}.
\]
By convention, $g(a) = f(\sigma_0)$ if $\widehat{H}_a$ is not null on $\Re(z) > -f(\sigma_0)$.
By definition of $\lambda^*_a$ and by the results of Section~\ref{sec:comportement en temps long avec drift constant} 
we know that
\(g(a)\in(0,\lambda^*_a]\).	
So, to prove the lemma, it suffices to show the following result\\
\textbf{Claim} $g$ is lower semi-continuous, that is
\[ 
\forall a_0 \in [0, \bar{a}], ~\liminf_{a \rightarrow a_0}{g(a)} \geq g(a_0).
\]

\noindent \textit{Proof of the claim.} Choose $a_0 \in [0, \bar{a}]$. 
We have $g(a_0) > 0$. Fix $\lambda \in (0 , g(a_0))$.
Thanks to Lemma~\ref{lem:les zeros sont dans un cone}, one can find \(R > 0\),
such that for all \(a\in[0,\bar{a}]\), for all \(z\) with \(\RePart(z)\in[-\lambda,0]\) and \(\ImPart(z)\notin [-R,R]\),
we have \(\widehat{H}_a(z)\neq 0\).
Denote \(U=\{z\in\mathbb{C}, \RePart(z)\in[-\lambda,0], |\ImPart(z)|\leq R \}\).
By definition of $g(a_0)$, we have $\widehat{H}_{a_0} \neq 0$ on $U$ and the continuity of $z \mapsto \widehat{H}_{a_0}(z)$ yields
\(\inf_{z \in U} { |\widehat{H}_{a_0}(z)|} > 0\).
Moreover, \((a,z) \mapsto \hat{H}_a(z)\) is
continuous on \([0,\bar{a}]\times U\), so one can find \(\delta > 0\) such that for all \(| a - a_0| \leq \delta\), \(z\in U\), we have \(|\widehat{H}_{a}(z)|\neq 0\).
and so $g(a) \geq \lambda$.
We have proved that \(\forall \lambda\in (0,g(a_0))\), \(\liminf_{a \rightarrow a_0} g(a) \geq \lambda\).
It ends the proof.
\end{proof}

\begin{proposition}[Whole-line Palay-Wiener, an extension]
Let $\bar{a} > 0$ and for all $a \in [0, \bar{a}]$, let $k_a \in L^1(\mathbb{R}, \mathbb{R})$. Assume that
\begin{enumerate}
\item $\exists \eta \in L^1(\mathbb{R},\mathbb{R}_+)$ \textit{s.t.}  $\forall a \in [0, \bar{a}]$, 
\(\forall 0<\epsilon< 1, \forall t \in \mathbb{R},  |k_a(t) - k_a(t-\epsilon)| \leq \epsilon \eta(t). \)
\item $\exists \theta \in L^1(\mathbb{R}, \mathbb{R}_+)$ \textit{s.t.} $\forall a \in [0, \bar{a}]$,
\(\forall t \in \mathbb{R}: |k_a(t)| \leq \theta(t)\).
\item $\forall a \in [0, \bar{a}]$, $\forall y \in \mathbb{R}$ let $\widehat{k}_a(i y) = \int_{\mathbb{R}}{e^{- i y t } k_a(t) dt}$. We assume that
\[ \inf_{a \in [0, \bar{a}], y \in \mathbb{R}}{|1-\widehat{k}_a(iy)|} > 0. \]
\end{enumerate}
Then for all $a \in [0, \bar{a}]$, there exists a function $x_a \in L^1(\mathbb{R}, \mathbb{R})$ satisfying the equation $x_a = k_a + k_a * x_a$ and
 \[ \sup_{a \in [0, \bar{a}]}{ ||x_a||_{L^1}} < \infty. \]
\label{prop:whole line palay wiener theorem extension}
\end{proposition}
\begin{proof}
We follow the proof of Theorem 4.3 in \cite[Chap. 2]{gripenberg_volterra_1990} and emphasis on the differences.
Let $\zeta(t) := \frac{1}{\pi t^2}(1 - \cos(t))$ be the Fejer kernel;
its Fourier transform is $\widehat{\zeta}(iy) = (1 - |y|) \indic{|y| \leq 1}$.
For any $p \geq 1$, set $\zeta_p(t) := p \zeta (p t)$ and 
\(\forall a \in [0, \bar{a}]\),
\[
k^\infty_a(t) := k_a - \zeta_p * k_a.
\]
\textbf{Claim 1} There is an integer $p > 0$ such that $\forall a \in [0, \bar{a}]$, \(\forall |y| \geq p\), we have
\[  
||k^\infty_a||_{L^1} \leq 1/2 \quad\text{ and } \quad \widehat{k}^\infty_a(iy) =\widehat{k}_a(iy).   
\]

\noindent \textit{Proof of the claim.} It is clear that with this choice of $\zeta$, $\forall |y| \geq p$, $\widehat{k}^\infty_a(iy) = \widehat{k}_a(iy)$.
Moreover, using $\int_{\mathbb{R}}{\zeta_p(s) ds } = 1$, we have
\begin{align*}
||k^\infty_a||_{L^1} &= \int_{\mathbb{R}}{ \left|\int_{\mathbb{R}} {k_a(t) \zeta_p(s) - k_a(t-s) \zeta_p(s) ds} \right| dt} \\
& \leq \int_{\mathbb{R}}{ \zeta(u) \int_{\mathbb{R}}{ |k_a(t) - k_a(t - \frac{u}{p})|}  dt du}.
\end{align*}
We used the Tonelli-Fubini Theorem (everything is non-negative).
Let $R> 0$ such that $\int_{\mathbb{R} \setminus [-R, R]}{ \zeta(u) du } \leq \frac{1}{8 ||\theta||_{L^1}}$.
It follows that
\begin{align*}
||k^\infty_a||_{L^1}  & \leq  1/4 + \int_{-R}^R{ \zeta(u) \int_\mathbb{R}{|k_a(t) - k_a(t - \frac{u}{p})| dt} du} \\
& \leq 1/4 + \int_{-R}^R{ \left( \int_\mathbb{R}{|\frac{u}{p}| \eta(t) dt}  \right) du} \\
& \leq 1/4  + \frac{R^2}{p} ||\eta||_{L^1}.
\end{align*}
The claim is proved by choosing an integer $p \geq 4 R^2 ||\eta||_{L^1}$.

Along the same idea, we define $\beta(t) := 4 \zeta(2t) - \zeta(t) = \frac{1}{\pi t^2} (\cos{t} - \cos{2t})$. Note that \(\forall |y|\leq 1\),
we have \(\widehat{\beta}(iy) = 1\). Then for all $\delta > 0$, we set $\beta_{\delta}(t) = \delta \beta(\delta t)$ and 
\[ 
\forall y_0 \in \mathbb{R},\forall t \geq 0,\quad k^{y_0, \delta}_a(t) = \int_{\mathbb{R}}{ (\beta_\delta(t-s)  - \beta_\delta(t)) e^{i y_0 (t-s)} k_a(s)  ds} .
\]
\textbf{Claim 2} Given $\epsilon > 0$, one can find a constant $\delta > 0$ such that: $\forall y_0 \in \mathbb{R}, \forall a \in [0, \bar{a}]$, 
\[ 
\forall |y-y_0| \leq \delta,~ \widehat{k}_a(i y) = \widehat{k}_a(i y_0) + \widehat{k^{y_0, \delta}_a}(i y)
 \quad  \text{ and } \quad 
 ||k^{y_0, \delta}_a||_{L^1} \leq \frac{\epsilon}{2}.
\] 
\textit{Proof of the claim.} By definition of $k^{y_0, \delta}_a$ it holds that
\[ \forall y \in \mathbb{R},  \widehat{k^{y_0, \delta}_a}(i y)  = \widehat{\beta}_\delta(i (y - y_0))(\widehat{k}_a(i y) - \widehat{k}_a(i y_0)). \]
Moreover, $\widehat{\beta}_\delta(i y) = 1$ if $|y| \leq \delta$ and consequently the first point of the claim is satisfied.
Furthermore, 
\begin{align*}
 \int_\mathbb{R} { |k^{y_0, \delta}_a(t)| dt } & \leq \int_{\mathbb{R}}{ |k_a(s)| \int_{\mathbb{R}}{ |\beta(t - \delta s) - \beta(t)| dt } ds } \\
 & \leq \int_{\mathbb{R}}{ \theta(s) \int_{\mathbb{R}}{ |\beta(t - \delta s) - \beta(t)| dt } ds }. 
\end{align*} 
The right hand side does not  depend on $y_0$ nor $a$ and goes to zero as $\delta$ goes to zero. 
This proves the second point of the claim.

It follows from Claim~1 that $\forall a \in [0, \bar{a}]$, the equation $x^\infty_a = k_a + k^\infty_a * x^\infty_a$ has a unique solution $x^\infty_a \in L^1(\mathbb{R})$ with $||x^\infty_a||_{L^1} \leq 2 ||\theta||_{L^1}$.  Moreover, we have
\[ 
\forall a \in [0, \bar{a}], \forall |y| \geq p,\quad \widehat{x}^\infty_a(iy) = \frac{\widehat{k}_a(iy) }{1 - \widehat{k}_a(iy)}. 
\]
Similarly, we define $\epsilon := \inf_{a \in [0, \bar{a}], y \in \mathbb{R}}{|1-\widehat{k}_a(iy)|} > 0$ and apply the second claim.
Given $y_0 \in \mathbb{R}$ and $a \in [0, \bar{a}]$, let $A^{y_0}_a = \frac{1}{1 - \widehat{k}_a(iy_0)}$.
We have  
$1-\widehat{k}_a(iy) = 1 - \widehat{k}_a(iy_0) - \widehat{k^{y_0, \delta}_a}(i y) = \frac{1}{A^{y_0}_a} (1 - A^{y_0}_a \widehat{k^{y_0, \delta}_a}(i y))$.
So,
\[ 
\forall |y-y_0| \leq \delta, \quad  \frac{\widehat{k}_a(iy)}{1-\widehat{k}_a(iy)} = \frac{ A^{y_0}_a \widehat{k}_a(iy)}{1- A^{y_0}_a\widehat{k^{y_0, \delta}_a}(i y) }.  
\]
Using  $||A^{y_0}_a k^{y_0, \delta}_a||_{L^1} \leq 1/2 $, we can
define 
the solution of $x^{y_0}_a = A^{y_0}_a k_a + A^{y_0}_a k^{y_0, \delta}_a * x^{y_0}_a$ and we have
\[ 
||x^{y_0}_a||_{L^1} \leq \frac{2}{\epsilon} ||\theta||_{L^1}.  
\]
Consequently, for all $y$ with $|y - y_0| \leq \delta$ we have
\[ \widehat{x}^{y_0}_a(i y) = \frac{\widehat{k}_a(iy)}{1-\widehat{k}_a(iy)}. \]
Furthermore, still following \cite{gripenberg_volterra_1990}, one can find an integer $m >0$ such that: $ \forall a \in [0, \bar{a}]$, 
$\forall j \in \mathbb{Z}, |j|\leq  m  p$, 
there exists a function $x^{j/m}_a \in L^1(\mathbb{R})$ with  $||x^{j/m}_a||_{L^1} \leq \frac{2}{\epsilon} ||\theta||_{L^1}$ such that
\[ 
\forall |y - j/m| \leq 1/m,\quad \widehat{x^{j/m}_a}(iy) = \frac{\widehat{k}_a(iy)}{1-\widehat{k}_a(iy)}. 
\]
We define $\psi_j(t) = \frac{1}{m} e^{-i j t / m} \zeta(t/m)$. We have
\(||\psi_j||_{L^1} = 1\). Its Fourier transform is given by
\[
\widehat{\psi}_j(iy) =
\left\{
\begin{array}{ll}
0 & \text{ if }  |y-j/m|>1/m\\
1 - m |y-j/m| & \text{ otherwise. }
\end{array}
\right.
\]
We set
\[ 
x_a = \sum_{|j| \leq mp}{\psi_j * (x^{j/m}_a - x^\infty_a)} + x^\infty_a.
\]
It is clear that $x_a \in L^1(\mathbb{R})$ and that
\[ 
\sup_{a \in [0, \bar{a}]}{ ||x_a||_{L^1}}  \leq   m  p \left(\frac{2}{\epsilon} ||\theta||_{L^1}  + 2 ||\theta||_{L^1}\right) + 2 ||\theta||_{L^1}< \infty. \]
With this choice of $\psi_j$, $\forall y \in \mathbb{R}, \widehat{x}_a(iy) = \frac{\widehat{k}_a(iy)}{1-\widehat{k}_a(iy)}$ and by uniqueness of the Fourier transform, we conclude that $x_a$ is the solution of $x_a = k_a + k_a * x_a$. 
\end{proof}
As a consequence of the previous result, we have
\begin{corollary}
Let $\bar{a} > 0$, define $\lambda^* = \inf_{a \in [0, \bar{a}]}{\lambda^*_a}$ ($\lambda^* > 0$ by Lemma~\ref{lem: le min des lambda star}). Let $0 < \lambda < \lambda^*$ and consider $\r[][a]$ the solution of the Volterra equation $\r[][a] = \K[][a] + \K[][a] * \r[][a]$.
By Proposition~\ref{prop: comportement en temps long equation de volterra cas convolutionnel}, it holds that $\r[][a] = \gamma(a) + \xi_a$ for some $\xi_a \in L_\lambda$.
Then we have $\sup_{a \in [0, \bar{a}]}{||\xi_a||_{\lambda,1}} < \infty.$
\label{lem:xi_a est uniformement borne}
\end{corollary}
\begin{proof}
	Recall (see proof of Proposition~\ref{prop: comportement en temps long equation de volterra cas convolutionnel}) that $\xi_a(t) = e^{-\lambda t} \xi_{a,-\lambda}(t)$ and so \(||\xi_a||_{\lambda,1} = ||\xi_{a,-\lambda}||_{L^1}\). We now prove that Proposition~\ref{prop:whole line palay wiener theorem extension} applies to \(\xi_{a,-\lambda}\).
	Indeed, it solves
\[ 
\xi_{a,-\lambda} = \K[][a,-\lambda]  + \K[][a,-\lambda] * \xi_{a,-\lambda}, 
\]
with $\K[][a,-\lambda](t) := e^{\lambda t} \K[][a](t) \indic{t \geq 0}$. 
It remains to show that \(\K[][a,-\lambda]\) fulfills the assumptions of Proposition~\ref{prop:whole line palay wiener theorem extension}.

\noindent 1. We use \(\sup_{a \in [0, \bar{a}]}K_a(t) \leq f(\varphi_t^{\bar{a}}(0))H_0(t)\) and \(\sup_{a \in [0, \bar{a}]}
|\varphi_t^{a}(0) - \varphi_{t-\epsilon}^{a}(0)|\leq \epsilon C_b^{\bar{a}}.
\)

\noindent 2. For all \(t\geq 0\) and \(a\in[0,\bar{a}]\), we have
\[\K[][a,-\lambda](t) \leq \theta(t) := e^{\lambda t} f(C^{\bar{a}}_t) \H[][0](t) \indica{\mathbb{R}_+}(t) \in L^1(\mathbb{R}). \]

\noindent 3. We have \(\widehat{\K[][a,-\lambda]}(iy) = \widehat{K}_a(-\lambda + iy)\). We conclude by Lemmas~\ref{lem: le min des lambda star} and \ref{lem:les zeros sont dans un cone}.
\end{proof}
\subsection{Proof of Theorem\protect~\ref{th:Jpetit}}

We are now ready to give the proof of the main theorem.
\begin{itemize}
\item \textbf{Step 1}  Recall that equation \eqref{eq:20190424} gives 
\[ 
\frac{d}{dt} \E f(X_t) \leq \frac{1}{2}[\bar{r}(J) ^2 - \E^2 f(X_t) ], 
\]
where $(X_t)_{t \geq 0}$ is the solution of the non-linear equation \eqref{NL-equation0} and the function \(J \mapsto \bar{r}(J)\) is non-decreasing. 
Using Proposition~\ref{prop: la constant bar(a)} with $\kappa := J \bar{r}(J) + 1$, there is a non-decreasing function $J \mapsto \bar{a}(J)$ such that: 
\[  \forall J,s \geq 0,~ \forall (a_t)_{t \geq s} \in \mathcal{C}([s, \infty), \mathbb{R}_+),~ [\sup_{t \geq s}{a_t }\leq \bar{a}(J) \text{ and } J \nu(f) \leq \bar{a}(J)] \implies \sup_{t \geq s}{J \r[\nu][(a.)](t, s)} \leq \bar{a}(J) . \]
Moreover, it holds that $\forall J \geq 0,~ J \bar{r} (J) < \bar{a}(J)$.
\item \textbf{Step 2} We define
\[ 
\lambda^* := \inf_{a \in [0, \bar{a}(J_m)]}{\lambda^*_a}, 
\]
where $J_m > 0$ is defined in Proposition~\ref{prop:mesures invariantes}.
Lemma~\ref{lem: le min des lambda star} gives $\lambda^* > 0$. We now fix $\lambda$ such that $0 < \lambda < \lambda^*$. 
\item \textbf{Step 3}
\begin{itemize}
\item Using Corollary~\ref{lem:xi_a est uniformement borne}, we know that the solution of the Volterra equation $\r[][a] = \K[][a] + \K[][a] * \r[][a]$ is $\r[][a] = \gamma(a) + \xi_a$ with $\xi_a \in L_\lambda$ and that:
\[ \xi^\infty(J)  := \sup_{a \in [0, \bar{a}(J)]} || \xi_a||_{\lambda,1}< \infty.\]
It is clear that $J \mapsto \xi^\infty(J)$ is non-decreasing (as $J \mapsto \bar{a}(J)$ is).
\item One can find a function $k^\infty: \mathbb{R}_+ \times \mathbb{R}_+ \rightarrow \mathbb{R}_+$, non-decreasing with respect to its two parameters, such that for all $(a_t) \in \mathcal{C}(\mathbb{R}_+, \mathbb{R}_+)$ we have:
\[ \sup_{t \geq 0}{a_t} \leq \bar{a} \implies ||\K[\nu][(a.)]||_{\lambda,\infty} \leq k^\infty(\nu(f), \bar{a}) < \infty. \]
Moreover, one can find a constant $h^\infty$ (only depending on $\lambda$, $b$ and $f$) such that for all $(a_t) \in \mathcal{C}(\mathbb{R}_+, \mathbb{R}_+)$, we have
\[ 
||\H[\nu][(a.)]||_{\lambda,\infty} \leq h^\infty. 
\]
These two points follow from $\lambda < f(\sigma_0)$,
Assumption~\ref{assumptions:b0}, Remarks~\ref{remark:b0f}(\ref{remark:b0f:lim f flow a}) and \ref{remark:inequality Ha}.
\item The function $\eta_{\bar{a}}$ of Lemma~\ref{lem: la fonction eta qui controle K bar(a) et H bar(a)} satisfies
\[ 
||\eta_{\bar{a}}||_1 < \infty, ~\bar{a} \mapsto ||\eta_{\bar{a}}||_1 \text{ is non-decreasing, }
\]
and consequently the function $J \mapsto  ||\eta_{\bar{a}(J)}||_1$ is non-decreasing.
\item Finally the normalization $\gamma$ is a non-decreasing function of \(a\) (see \eqref{eq:equation gamma a avec le flot}) and it follows that 
\[ 
\forall a \in [0, \bar{a}(J)], \gamma(a) \leq \gamma(\bar{a}(J)). 
\]
\end{itemize}
\item \textbf{Step 4} Let $\nu$ be a probability measure such that $\nu(f) \leq \bar{r}(J_m) +1$. 
Remind that  for all $J \in (0,J_m)$ the equation $a \gamma^{-1}(a) = J$ has a unique solution $a^*(J) \in [0, \bar{a}(J_m)]$. We now apply 
Proposition~\ref{prop: theoreme general de perturbation avec les constantes explicites} with $\alpha = 1/2$.
Define:
\begin{align*}  C(J) & := \frac{1}{2 ||\eta_{\bar{a}(J)}||_1 (1 + \xi^\infty(J) + \gamma(\bar{a}(J)))}\\
 D(J) & := 2(1 + \gamma(\bar{a}(J)) + \xi^\infty(J)) k^\infty(\bar{r}(J_m) +1, \bar{a}(J)) + \gamma(\bar{a}(J)) h^\infty. 
 \end{align*}
From Step~3, it is clear that the functions $J \mapsto \tfrac{1}{C(J)}$ and $J \mapsto D(J)$ are non-decreasing. Consequently, we can find a constant $J^* \in (0, J_m)$ such that
\[ 
\forall J \in [0, J^*],\quad \frac{J D(J)}{C(J)} \leq 1.
\]
Proposition~\ref{prop: theoreme general de perturbation avec les constantes explicites} tells us that for every $0 \leq J \leq J^*$, given any $(a_t)_{t\geq 0} \in \mathcal{C}(\mathbb{R}_+, \mathbb{R}_+)$ with $\sup_{t \geq 0}{a_t} \leq \bar{a}(J)$ and such that
\[ 
\forall t \geq 0,~|a_t - a^*(J)| \leq C(J) e^{-\lambda t}, 
\]
it holds
\[ 
\forall t \geq 0,~ |J \r[\nu][(a.)](t,0) - a^*(J) | \leq C(J) e^{- \lambda t}. 
\]
\item \textbf{Step 5} Let now $J \in (0, J^*]$ be fixed (the case $J = 0$ is already treated by Proposition~\ref{prop: convergence du taux de saut dans le cas J=0}). 
We assume the initial condition $\nu$ of \eqref{NL-equation0} satisfies $J \nu(f) \leq \bar{a}(J)$ and that $\nu(f) \leq \bar{r}(J_m) + 1$ (we shall come back to the general case in Step 6).
We define recursively $a^n \in \mathcal{C}(\mathbb{R}_+, \mathbb{R}_+)$ by
\[ 
\forall t \geq 0,~a^0(t) := a^*(J)\quad \mbox{ and } \quad\forall n\geq 0, \quad ~a^{n+1}(t) := J \r[\nu][(a^n.)](t, 0). \]
By Step~4 and by induction, it holds that:
\[ \forall n \geq 0,~\forall t \geq 0,~|a^n (t) - a^*(J) | \leq C(J) e^{- \lambda t}. \]
We deduce that:
\begin{align*} 
\forall t \geq 0,~|\E f(X_t) - \gamma(a^*(J))| & \leq |\E f(X_t) - \r[\nu][(a^{n}.)](t, 0)| + \frac{1}{J}|a^{n+1}(t) - a^*(J)| \\
& \leq   \frac{1}{J}|J \E f(X_t) - a^{n+1}(t)| + \frac{C(J)}{J}e^{-\lambda t}. 
\end{align*}
The Picard iteration studied in Part~\ref{section: existence de l'equation non lineaire} shows that
\[ 
\forall t \geq 0,~ \lim_{n \rightarrow \infty}{ |J \E f(X_t) - a^n(t)| = 0}.
\]
We have proved that
\[ 
\forall t \geq 0,~  |\E f(X_t) - \gamma(a^*(J))| \leq \frac{ C(J)}{J} e^{-\lambda t}.  
\]
\item \textbf{Step 6} We now prove that there exists \(s\geq 0\) such that $\E f(X_s) \leq \min(\frac{\bar{a}(J)}{J}, \bar{r}(J_m) + 1)$.
By Step~1, we have $\limsup \E f (X_t) \leq \overline{r}(J)$.
Since $\overline{r}(J)<\overline{a}(J) / J$ and since $\overline{r}(J) \leq \overline{r}\left(J_{m}\right)$, the conclusion follows. Consequently, Step 5 can be applied to the process $(X_t)_{t \geq s}$ starting with $\nu  = \mathcal{L}(X_s).$ This proves the convergence of the jump rate. 
\end{itemize}
The convergence of the law of $X_t$ to the invariant measure then follows from Proposition~\ref{prop:convergence en loi vers la mesure invariante}. 
This ends the proof of Theorem~\ref{th:Jpetit}.
\begin{remark}
There is some freedom in the above construction of the constants $\lambda$ and $J^*$. We can 
choose any $\lambda$ in $[0, \lambda^*)$ and the value of $J^*$ depends both on $\lambda$ and 
on a parameter $\alpha \in (0,1)$, here chosen to be equal to $1/2$ (see Step~4). We may 
optimize this construction to get either $J^*$ or $\lambda$ as large as possible.
\end{remark}

\section*{Acknowledgements}
The authors warmly thank one anonymous referee for his/her deep reading of a previous version of the paper. His/Her comments have strongly helped us to improve  the presentation.

\noindent This project/research has received funding from the European Union’s Horizon 2020 Framework Programme for Research and Innovation under the Specific Grant Agreement No. 785907 (Human Brain Project SGA2).

\bibliographystyle{abbrv}

\end{document}